\documentclass[a4paper,10pt]{amsart}
\usepackage[utf8]{inputenc}
\usepackage{amssymb,amsfonts,amsmath,mathtools,mathrsfs}
\usepackage{enumerate,setspace,bm}
\usepackage{graphicx}
\usepackage{xcolor}
\usepackage{pgf,tikz}
\usetikzlibrary{arrows}
\usetikzlibrary[patterns]
\usepackage{thmtools} 
\usepackage[foot]{amsaddr}

\usepackage[colorinlistoftodos,prependcaption,color=yellow,textsize=tiny]{todonotes}

\usepackage[pdfdisplaydoctitle,colorlinks,breaklinks,urlcolor=blue,linkcolor=blue,citecolor=blue]{hyperref} 

\newcommand{\F}{\mathcal{F}}

\newcommand{\N}{\mathbb{N}}

\newcommand{\PP}{\mathcal{P}}

\newcommand{\R}{\mathbb{R}}

\newcommand{\T}{\mathcal{T}}

\newcommand\Z{\mathbb{Z}}

\renewcommand{\epsilon}{\varepsilon}

\newcommand{\set}[1]{\left\{#1\right\}}
\newcommand{\pa}[1]{\left(#1\right)}

\newtheorem{theorem}{Theorem}

\newtheorem{corollary}[theorem]{Corollary}
\newtheorem{definition}[theorem]{Definition}
\newtheorem{lemma}[theorem]{Lemma}
\newtheorem{proposition}[theorem]{Proposition}
\newtheorem*{proposition*}{Proposition}

\theoremstyle{remark}
\newtheorem{remark}[theorem]{Remark}

\title[Intrinsic Randomness in PDEs]{An example of intrinsic randomness in deterministic PDEs}
\author[F. Flandoli]{Franco Flandoli}
\address{Scuola Normale Superiore, Piazza dei Cavalieri, 7, 56126 Pisa, Italia}
\email{\href{mailto:franco.flandoli@sns.it}{franco.flandoli@sns.it}}

\author[B. Gess]{Benjamin Gess}
\address{Max Planck Institute for Mathematics in the Sciences, Leipzig, Germany \newline \& Fakult\"at f\"ur Mathematik, Universit\"at Bielefeld, Bielefeld, Germany}
\email{\href{mailto:benjamin.gess@mis.mpg.de}{benjamin.gess@mis.mpg.de}}

\author[F. Grotto]{Francesco Grotto}
\address{Université du Luxembourg, 6 Avenue de la Fonte, 4364 Esch-sur-Alzette, Luxembourg}
\email{\href{mailto:francesco.grotto@uni.lu}{francesco.grotto@uni.lu}}
\date\today

\begin{document}

\begin{abstract}
  A new mechanism leading to a random version of Burgers' equation is introduced: 
  it is shown that the Totally Asymmetric Exclusion Process in discrete time (TASEP) 
  can be understood as an intrinsically stochastic, non-entropic weak solution of Burgers' equation on $\R$. 
  In this interpretation, the appearance of randomness in the Burgers' dynamics is caused by random additions of jumps to the solution, corresponding to the random effects in TASEP.
%
%
\end{abstract}

\maketitle

\section{Introduction}

Random solutions and stochastic versions of the Burgers' equation
\begin{equation}\label{eq:burgers}
\partial_{t}u(t,x)+u(t,x)\partial_{x}u(t,x)=0,\quad u(0,x)=u_0(x), \quad x\in\R,t\geq 0,
\end{equation}
appear in various contexts, forms and applications. 
Relevant examples include \cite{zel1970gravitational,gurbatov2012large}, where the
(multi-dimensional) Burgers' equation with Gaussian initial conditions was found in the study of the formation of large-scale structures in the Universe, \cite{kardar1986dynamic} where the Burgers' equation with a random forcing arises in the analysis of the dynamics of interfaces and in \cite{Lions2013}, where the Burgers' equation with random flux appears in the analysis of mean field systems with common noise, 
related to mean field game systems.

In the present work, we uncover another mechanism along which randomness can enter the dynamics of the Burgers' equation,
by establishing a one-to-one correspondence between a certain class of solutions to Burgers' equation
and the discrete-time totally asymmetric simple exclusion process (TASEP):
Sample functions are shown to be random (weak) solutions to the Burgers' equation, stochasticity being introduced 
by a random creation of jumps in the solution, corresponding to jumps of TASEP particles.

Different concepts of solutions to Burgers' equation \eqref{eq:burgers} are known, 
which is due to the fact that weak solutions are non-unique. A unique characterization of weak solutions in form of entropy solutions can be given in the setting of vanishing diffusion approximations, 
that is, in the case that solutions to \eqref{eq:burgers} are obtained as limits for  ($\varepsilon\downarrow0$) of
\begin{equation}\label{eq:viscous-burgers}
\partial_{t}u(t,x)+u(t,x)\partial_{x}u(t,x)=\varepsilon \partial_{xx} u(t,x).
\end{equation}
In contrast, in the setting of vanishing diffusion-dispersion approximations ($\varepsilon,\delta\downarrow0$)
\begin{equation}\label{eq:diffusion-dispersion}
\partial_{t}u(t,x)+u(t,x)\partial_{x}u(t,x)=\varepsilon \partial_{xx} u(t,x)+\delta \partial_{xxx} u(t,x)
\end{equation}
with scaling $\delta \approx \varepsilon^2$ non-classical shocks 
(\emph{i.e.} violating entropy conditions for certain entropies) are known to appear
\cite{Hayes1997,LeFloch1999}.
The characterization of the class of weak solutions produced in this setting is an open problem. 
Similarly, limits of relaxation approximations \cite{Hwang2002} to the Burgers' equation are known to converge 
to weak solutions of \eqref{eq:burgers}, which are only known to be so-called quasi-solutions
\cite{DeLellis2003}, \emph{i.e.} which have finite, but not necessarily signed entropy production.
Finally, in the last section of this work we comment on how the Burgers' equation heuristically appears to be  
related to the KPZ fixed point \cite{Matetski2016}. Also in this setting, the correct concept of solutions does not seem to be entropy solutions, see \cite[Equation (1.3) ff.]{Matetski2016} and their identification is an open problem.
The concept of a solution to \eqref{eq:burgers} thus depends on the underlying application.

We will discuss how one can exploit non-uniqueness of weak solutions to perform
\emph{random} choices when extending discrete time dynamics to weak solutions to the Burgers' equation in continuous time, producing a stochastic process whose trajectories
are non-entropic weak solutions to \eqref{eq:burgers}, a so-called \emph{intrinsically random solution}.
Intrinsic stochasticity, that is, stochastic solutions of deterministic differential equations 
with deterministic initial conditions, is an interesting and challenging concept,
for example arising in turbulence theory.
For results and discussions about this notion we refer to
\cite{Weinan2000,Weinan2000a,Kupiainen2003,LeJan2002,Mailybaev2016,Drivas2020,thalabard2020butterfly}.

Discrete-time TASEP consists in particles occupying sites of $\Z$, jumping at random to their right
under the constraint that each site may not be occupied by more than one particle at once.
We will define a closely related discrete-time particle model on $\Z$, which we call Active Bi-Directional Flow (ABDF),
starting from TASEP and considering pairs of occupied and empty positions.
ABDF model consists in particles constantly moving to their left or right,
annihilating in pairs when colliding and being generated \emph{in pairs, at random}, in certain positions.
We will show that it is conjugated to TASEP as random dynamical systems in \autoref{thm:conjugacy}.
The ABDF model shares features with other ones related to TASEP and the \emph{KPZ universality class} to which the latter belongs,
in particular the discrete-time polynuclear growth (PNG) process; however, to the best of our knowledge, the construction is original.

The behaviour of particles of ABDF model can be precisely mirrored by particular weak solutions of \eqref{eq:burgers}
composed of indicator functions of intervals, which we call \emph{quasi-particles}, 
traveling to their left or right following characteristic lines
until two of them meet. It is when quasi-particle collide that non-uniqueness is exploited to annihilate them,
and also, by time-reversal, the same can be done to generate pairs of quasi-particles out of a null profile.
The main result, \autoref{thm:mainresult}, consists in showing that this close analogy between ABDF model and a random selection
of Burgers' weak solutions can be made precise with a bijection between samples of models.

Aside from the interest of ``embedding'' discrete-time random processes into weak solutions of Burgers' equation,
this study stems from an attempt of understanding possible links between non-entropic solutions
of Burgers' equation and the aforementioned KPZ universality class and KPZ fixed points.
Hence the particular choice of TASEP as the ``source'' of intrinsic randomness,
it being a most distinguished model in the study of KPZ universality.
At this stage, what we can state to that end remains essentially conjectural,
so we collect related observations and references to the last Section of the article.

The paper is organized as follows.
In \autoref{sect ABDF} we introduce the ABDF model; in Section
\ref{sect Tasep and ABDF} we link it to TASEP and finally, in Section
\ref{Sect Burgers TASEP ABDF} we link them to the Burgers' equations.
Preliminarily, in \autoref{Sect Burgers}, we introduce the class of weak
solutions of Burgers' equations involved in the conjugacy. Some ideas about
such solutions have been identified previously \cite{Boritchev}, but the link
with TASEP described here is new.

\section{ABDF model: Active Bi-Directional Flow\label{sect ABDF}}

We begin with an informal description:
a configuration of the ABDF model is made of particles and empty positions on $\Z$,
with no more than one particle at each position. Particles are divided into
two classes: \textit{left and right particles}, according to the direction in which they are allowed to move.
Empty positions are also divided into two classes, \textit{active and inert positions},
the former being allowed to generate couples of new particles as we will detail.

Relative positions of left and right
particles is not arbitrary:  two particles are \textit{consecutive}
if they occupy positions $x_1<x_2$ such that no
particle is in between, and we postulate
that two consecutive particles of different type (one left and one right,
independently of the order) are always separated by an odd number of empty
positions. Moreover two consecutive particles of the same kind shall 
always be separated by an even number of empty positions.

Empty positions are active or inert depending on their
distance from the first particle on their left or right, and the class of the latter. 
Precisely, assume the empty position, say $x_0$, lies between consecutive particles at $x_1<x_2$. 
Let $k=x_0-x_1$; if the particle at $x_1$ is a left-particle and $k$ is odd, then the empty
position at $x_0$ is active, otherwise it is inert. If the particle at
$x_1$ is right and $k$ is odd, then the empty position is inert, otherwise
it is active. The same definition is given in terms of $x_2$:
if $h=x_2-x_0$ is odd and the particle at $x_2$ is of left-type,
then the empty position at $x_0$ is inert; if the particle
at $x_2$ is of right-type and $h$ is odd, then it is active, otherwise the empty
position is inert. It is easy to see that the two definitions coincide.

Finally, if an empty position is not between consecutive particles, 
either that we are dealing with an empty configuration,
or such position is part of an half line of empty particles. In the second case, the
rule concerning active or inert property is the same described above.
The case of all empty positions is a very special one:
active and inert positions should alternate, but they can do so in two ways, depending on the type of $x=0$. It is
important to distinguish between them, both possibly occurring during the
evolution described in the next \autoref{def ABDF dynamics}. For $x\in\Z$, we denote
\begin{equation*}
	alt_0(x)=x\, \mbox{mod }2,\quad alt_1(x)=(x+1)\, \mbox{mod }2,
\end{equation*}
and the double sequences
\[
\overline{alt}_{\alpha}(x)  =\left(  0,alt_{\alpha}\left(
x\right)  \right)  ,\qquad\alpha=0,1,
\]
the first coordinate declaring that $\overline
{alt}_{0}$ and $\overline{alt}_1$ are both zero sequences, namely they
represent ABDF configurations with all empty sites.

For an empty position, being active or inert is a (nonlocal)
\textit{consequence} of the particle configuration. Therefore, in the formal
definition of ABDF configurations we specify positions of particles  --first component of
the configuration-- and \textit{deduce} active or inert empty positions --second
component of the configuration-- by what we call \textit{activation map}. The
only exceptions are the ABDF configurations with all empty sites, where two
different activation profiles are possible.

Let us move to the rigorous definition. We associate $+1$ to right particles, $-1$
to left particles, $0$ to empty positions; then we introduce the
``activation record'' which associates $0$ to
any position where a new pair of particles cannot
arise (empty inert positions and occupied positions), and 1 to
active empty positions.

\begin{definition}\label{def ABDF config preliminary}
	Let $\Lambda_{0}$ be the set of sequences
	\begin{equation*}
		\theta:\Z\rightarrow\left\{  -1,0,1\right\}
	\end{equation*}
	which are not identically zero (we write simply $\theta\neq0$) such that
	\begin{itemize}
		\item[i)] if $x_1<x_2\in\Z$ have the properties $\theta\left(
		x_1\right)  \theta(x_2)  =-1$ and $\theta(x)
		=0$ for all $x\in\left(  x_1,x_2\right)  \cap\Z$, then the
		cardinality of $\left(  x_1,x_2\right)  \cap\Z$ is odd;
		\item[ii)] if $x_1<x_2\in\Z$ have the properties $\theta\left(
		x_1\right)  \theta(x_2)  =1$ and $\theta(x)  =0$
		for all $x\in\left(  x_1,x_2\right)  \cap\Z$, then the cardinality
		of $\left(  x_1,x_2\right)  \cap\Z$ is even.
	\end{itemize}
	For every $\theta\in\Lambda_{0}$, introduce the \emph{activation record} sequence
	\begin{equation*}
		ar\left(  \theta\right)  :\Z\rightarrow\left\{  0,1\right\},
	\end{equation*}
	defined as:
	\begin{itemize}
		\item[iii)] if $\theta\left(  x_0\right)  \in\left\{  -1,1\right\}  $ then
		$ar\left(  \theta\right)  \left(  x_0\right)  =0$;
		\item[iv)] if $\theta\left(  x_0\right)  =0$ and the set
		\[
		L\left(  x_0\right)  :=\left\{  x<x_0:x\in\Z\text{, }\theta\left(
		x\right)  \in\left\{  -1,1\right\}  \right\}
		\]
		is not empty, taking $x_1=\max L\left(  x_0\right)  $ and $k\in\N$
		such that $x_0=x_1+k$,
		\[
		ar\left(  \theta\right)  \left(  x_0\right)  =\left\vert \frac{\theta\left(
			x_1\right)  +(-1)  ^{k}}{2}\right\vert
		\]
		\item[v)] if $\theta\left(  x_0\right)  =0$ and the set
		\[
		R\left(  x_0\right)  :=\left\{  x>x_0:x\in\Z\text{, }\theta\left(
		x\right)  \in\left\{  -1,1\right\}  \right\}
		\]
		is not empty, taking $x_2=\min R\left(  x_0\right)  $ and $h\in\N$
		such that $x_0=x_2-h$,
		\[
		ar\left(  \theta\right)  \left(  x_0\right)  =\left\vert \frac{\theta\left(
			x_2\right)  +(-1)  ^{h+1}}{2}\right\vert .
		\]
	\end{itemize}
\end{definition}

The following results from a simple check.

\begin{lemma}\label{lemma coherence}
	If both $L\left(  x_0\right)  $ and $R\left(
	x_0\right)  $ are not empty, points (iv-v) above give the same definition
	of $ar\left(  \theta\right)  $.
\end{lemma}


\begin{definition}
\label{def ABDF config}A configuration of the ABDF model is a map%
\[
(\theta,act)  :\Z\rightarrow\left\{  -1,0,1\right\}
\times\left\{  0,1\right\}
\]
with the following properties:
\begin{itemize}
	\item[a)] if $\theta=0$, then either $act=alt_0$ or $act=alt_1$ (in other words,
	either $(\theta,act)  =\overline{alt}_{0}$ or $\left(
	\theta,act\right)  =\overline{alt}_1$);
	\item[b)] if $\theta\neq0$, then $\theta\in\Lambda_{0}$ and $act=ar\left(
	\theta\right)  $, where the set $\Lambda_{0}$ and the map $ar$ are introduced
	in \autoref{def ABDF config preliminary}.
\end{itemize}
The set of all ABDF configurations will be denoted by $\Lambda$.
\end{definition}

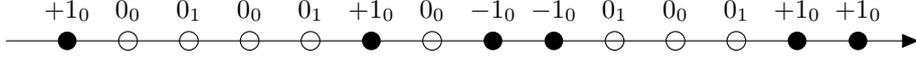
\begin{figure}
	\centering
	\begin{tikzpicture}[line cap=round,line join=round,>=triangle 45,x=0.8cm,y=0.8cm]
	\clip(1,0) rectangle (18,4);
	\draw [->] (1,1) -- (16,1);
	\fill [color=black] (2,1) circle (3.5pt);
	\draw (2,1.5) node {$+1_0$};
	\draw [color=black] (3,1) circle (3.5pt);
	\draw (3,1.5) node {$0_0$};
	\draw [color=black] (4,1) circle (3.5pt);
	\draw (4,1.5) node {$0_1$};
	\draw [color=black] (5,1) circle (3.5pt);
	\draw (5,1.5) node {$0_0$};
	\draw [color=black] (6,1) circle (3.5pt);
	\draw (6,1.5) node {$0_1$};
	\fill [color=black] (7,1) circle (3.5pt);
	\draw (7,1.5) node {$+1_0$};
	\draw [color=black] (8,1) circle (3.5pt);
	\draw (8,1.5) node {$0_0$};
	\fill [color=black] (9,1) circle (3.5pt);
	\draw (9,1.5) node {$-1_0$};
	\fill [color=black] (10,1) circle (3.5pt);
	\draw (10,1.5) node {$-1_0$};
	\draw [color=black] (11,1) circle (3.5pt);
	\draw (11,1.5) node {$0_1$};
	\draw [color=black] (12,1) circle (3.5pt);
	\draw (12,1.5) node {$0_0$};
	\draw [color=black] (13,1) circle (3.5pt);
	\draw (13,1.5) node {$0_1$};
	\fill [color=black] (14,1) circle (3.5pt);
	\draw (14,1.5) node {$+1_0$};
	\fill [color=black] (15,1) circle (3.5pt);
	\draw (15,1.5) node {$+1_0$};
	\end{tikzpicture}
	\caption{A piece of an ABDF configuration. Numbers $\pm 1$ and $0$ denote particle type or empty sites,
		subscripts are the values of activation record.}
	\label{fig:abdfconfig} 
\end{figure}

\autoref{fig:abdfconfig} represents a piece of an ABDF configuration. 
The example makes it apparent how the number of empty positions between non-empty ones 
is regulated by the concordance of signs of the extremes.

Let us come to the description of ABDF dynamics.
All right particles move to the right by one position at every time step, all
left particles to the left: unlike in exclusion processes, these jumps can not be
prevented by an occupied arrival positions, since \textit{all} particles move.
All active empty positions $x_0$ may generate, at random with probability
1/2, a pair of particles: a left-particle in $x_0-1$ and a right-particle at $x_0+1$.
It often happens that two particles meet at one position: a right particle
which moved from $x-1$ to $x$ and a left particle which moved from $x+1$ to
$x$ arrive at the same time $t$ at $x$. In such a case, the two particles
disappear, annihilating each other, and position $x$ becomes empty.

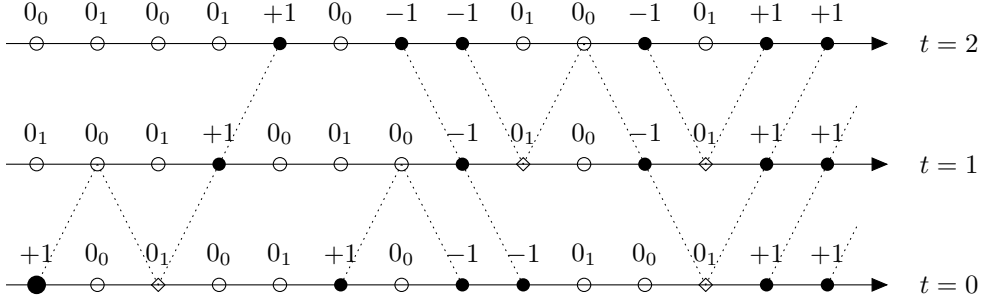
\begin{figure}
	\centering
	\begin{tikzpicture}[line cap=round,line join=round,>=triangle 45,x=0.8cm,y=0.8cm]
	\clip(1.5,0) rectangle (18,7);
	\draw [->] (1.5,1) -- (16,1);
	\draw [->] (1.5,3) -- (16,3);
	\draw [->] (1.5,5) -- (16,5);
	\draw [-,dotted] (7,1) -- (8,3);
	\draw [-,dotted] (9,1) -- (8,3);
	\draw [-,dotted] (10,1) -- (9,3);
	\draw [-,dotted] (14,1) -- (15,3);
	\draw [-,dotted] (15,1) -- (15.5,2);
	\draw [-,dotted] (4,1) -- (3,3);
	\draw [-,dotted] (4,1) -- (5,3);
	\draw [-,dotted] (13,1) -- (14,3);
	\draw [-,dotted] (13,1) -- (12,3);
	\draw [-,dotted] (2,1) -- (3,3);
	\draw [-,dotted] (5,3) -- (6,5);
	\draw [-,dotted] (9,3) -- (8,5);
	\draw [-,dotted] (10,3) -- (9,5);
	\draw [-,dotted] (10,3) -- (11,5);
	\draw [-,dotted] (12,3) -- (11,5);
	\draw [-,dotted] (14,3) -- (15,5);
	\draw [-,dotted] (15,3) -- (15.5,4);
	\draw [-,dotted] (13,3) -- (12,5);
	\draw [-,dotted] (13,3) -- (14,5);
	\fill [color=black] (2,1) circle (3.5pt);
	\fill [color=black] (7,1) circle (2.5pt);
	\fill [color=black] (9,1) circle (2.5pt);
	\fill [color=black] (10,1) circle (2.5pt);
	\fill [color=black] (14,1) circle (2.5pt);
	\fill [color=black] (15,1) circle (2.5pt);
	\draw [color=black] (3,1) circle (2.5pt);
	\draw [color=black] (4,1) ++(-2.5pt,0 pt) -- ++(2.5pt,2.5pt)--++(2.5pt,-2.5pt)--++(-2.5pt,-2.5pt)--++(-2.5pt,2.5pt);
	\draw [color=black] (5,1) circle (2.5pt);
	\draw [color=black] (6,1) circle (2.5pt);
	\draw [color=black] (8,1) circle (2.5pt);
	\draw [color=black] (11,1) circle (2.5pt);
	\draw [color=black] (12,1) circle (2.5pt);
	\draw [color=black] (13,1) ++(-2.5pt,0 pt) -- ++(2.5pt,2.5pt)--++(2.5pt,-2.5pt)--++(-2.5pt,-2.5pt)--++(-2.5pt,2.5pt);
	\draw [color=black] (3,3) circle (2.5pt);
	\draw [color=black] (8,3) circle (2.5pt);
	\fill [color=black] (9,3) circle (2.5pt);
	\fill [color=black] (15,3) circle (2.5pt);
	\fill [color=black] (5,3) circle (2.5pt);
	\fill [color=black] (14,3) circle (2.5pt);
	\fill [color=black] (12,3) circle (2.5pt);
	\draw [color=black] (2,3) circle (2.5pt);
	\draw [color=black] (4,3) circle (2.5pt);
	\draw [color=black] (6,3) circle (2.5pt);
	\draw [color=black] (7,3) circle (2.5pt);
	\draw [color=black] (10,3) ++(-2.5pt,0 pt) -- ++(2.5pt,2.5pt)--++(2.5pt,-2.5pt)--++(-2.5pt,-2.5pt)--++(-2.5pt,2.5pt);
	\draw [color=black] (11,3) circle (2.5pt);
	\draw [color=black] (13,3) ++(-2.5pt,0 pt) -- ++(2.5pt,2.5pt)--++(2.5pt,-2.5pt)--++(-2.5pt,-2.5pt)--++(-2.5pt,2.5pt);
	\fill [color=black] (6,5) circle (2.5pt);
	\fill [color=black] (8,5) circle (2.5pt);
	\fill [color=black] (9,5) circle (2.5pt);
	\draw [color=black] (11,5) circle (2.5pt);
	\fill [color=black] (15,5) circle (2.5pt);
	\draw [color=black] (2,5) circle (2.5pt);
	\draw [color=black] (3,5) circle (2.5pt);
	\draw [color=black] (4,5) circle (2.5pt);
	\draw [color=black] (5,5) circle (2.5pt);
	\draw [color=black] (7,5) circle (2.5pt);
	\draw [color=black] (10,5) circle (2.5pt);
	\fill [color=black] (12,5) circle (2.5pt);
	\draw [color=black] (13,5) circle (2.5pt);
	\fill [color=black] (14,5) circle (2.5pt);
	\draw (2,1.5) node {$+1$};
	\draw (3,1.5) node {$0_0$};
	\draw (4,1.5) node {$0_1$};
	\draw (5,1.5) node {$0_0$};
	\draw (6,1.5) node {$0_1$};
	\draw (7,1.5) node {$+1$};
	\draw (8,1.5) node {$0_0$};
	\draw (9,1.5) node {$-1$};
	\draw (10,1.5) node {$-1$};
	\draw (11,1.5) node {$0_1$};
	\draw (12,1.5) node {$0_0$};
	\draw (13,1.5) node {$0_1$};
	\draw (14,1.5) node {$+1$};
	\draw (15,1.5) node {$+1$};
	\draw (2,3.5) node {$0_1$};
	\draw (3,3.5) node {$0_0$};
	\draw (4,3.5) node {$0_1$};
	\draw (5,3.5) node {$+1$};
	\draw (6,3.5) node {$0_0$};
	\draw (7,3.5) node {$0_1$};
	\draw (8,3.5) node {$0_0$};
	\draw (9,3.5) node {$-1$};
	\draw (10,3.5) node {$0_1$};
	\draw (11,3.5) node {$0_0$};
	\draw (12,3.5) node {$-1$};
	\draw (13,3.5) node {$0_1$};
	\draw (14,3.5) node {$+1$};
	\draw (15,3.5) node {$+1$};
	\draw (2,5.5) node {$0_0$};
	\draw (3,5.5) node {$0_1$};
	\draw (4,5.5) node {$0_0$};
	\draw (5,5.5) node {$0_1$};
	\draw (6,5.5) node {$+1$};
	\draw (7,5.5) node {$0_0$};
	\draw (8,5.5) node {$-1$};
	\draw (9,5.5) node {$-1$};
	\draw (10,5.5) node {$0_1$};
	\draw (11,5.5) node {$0_0$};
	\draw (12,5.5) node {$-1$};
	\draw (13,5.5) node {$0_1$};
	\draw (14,5.5) node {$+1$};
	\draw (15,5.5) node {$+1$};
	\draw (17,1) node {$t=0$};
	\draw (17,3) node {$t=1$};
	\draw (17,5) node {$t=2$};
	\end{tikzpicture}
	\caption{A sample of ABDF dynamics starting from the configuration of \autoref{fig:abdfconfig}.
		Dotted lines track movements of particles. ``Activated'' empty sites have the empty circle replaced by an empty square.}
	\label{fig:abdfdynamics} 
\end{figure}

We have to check that these rules are coherent and that
they give rise to ABDF configurations described above.

\begin{definition}\label{def ABDF dynamics}Let $\Omega=\left\{  0,1\right\}  ^{\N%
\times\Z}$, with the $\sigma$-algebra $\mathcal{F}$ generated by
cylinder sets, and the product probability measure $P$ of Bernoulli
$p=\frac12$ random variables. Given $\omega\in\Omega$, we write
$\omega\left(  t,x\right)  $ for its $\left(  t,x\right)  $-coordinate,
$\left(  t,x\right)  \in\N\times\Z$ and write $\varkappa
\left(  t,x\right)  :=1-\omega\left(  t,x\right)  $ for the complementary value.

Then, based on the probability space $\left(  \Omega,\mathcal{F},P\right)  $,
we introduce a family of maps
\[
\T_{\text{ABDF}}\left(  t,\omega,\cdot\right)  :\Lambda
\rightarrow\left(  \left\{  -1,0,1\right\}  \times\left\{  0,1\right\}
\right)  ^\Z%
\]
indexed by $t\in\N$ and $\omega\in\Omega$, defined as follows. Denote
\[
\T_{\text{ABDF}}\left(  t,\omega,(\theta,act)  \right)
=\left(  \T_{\text{ABDF}}\left(  t,\omega,(\theta,act)
\right)  _1,\T_{\text{ABDF}}\left(  t,\omega,\left(  \theta
,act\right)  \right)  _2\right),
\]
where we recall that $act=ar\left(  \theta\right)  $ unless $\theta=0$.
The map $\T_{\text{ABDF}}\left(  0,\omega,\cdot\right)  $ is the
identity. For $t>0$, the first component is defined as
\begin{align}
\T_{\text{ABDF}}\left(  t,\omega,(\theta,act)  \right)
_1(x)   & :=\max\left(  \theta(x-1)  ,0\right)
+act(x-1)  \varkappa\left(  t-1,x-1\right) \label{ABDF map}\\
& +\min\left(  \theta(x+1)  ,0\right)  -act(x+1)
\varkappa\left(  t-1,x+1\right) \nonumber
\end{align}
for every $x\in\Z$. For $t>0$, if $\T_{\text{ABDF}}\left(
t,\omega,(\theta,act)  \right)  _1$ is not the identically null sequence, the second component is defined by
\[
\T_{\text{ABDF}}\left(  t,\omega,(\theta,act)  \right)
_2=ar\left(  \T_{\text{ABDF}}\left(  t,\omega,\left(
\theta,act\right)  \right)  _1\right)
\]
with $ar\left(  \cdot\right)$ as in \autoref{def ABDF config preliminary}. 
If $\T_{\text{ABDF}}\left(t,\omega,(\theta,act)  \right)  _1=0$ then $\T%
_{\text{ABDF}}\left(  t,\omega,(\theta,act)  \right)  _2$ is
either $alt_0$ or $alt_1$. It is equal to $alt_1$ in each one of the
following three cases:
\begin{gather*}
	\theta(0)=-1,\\
	act(-1)= 1 \quad \text{and} \quad \omega\left(  t-1,-1\right)= 1,\\
	act(0)=1 \quad \text{and} \quad \omega\left(  t-1,0\right)=0,
\end{gather*}
 otherwise it is equal to $alt_0$.
\end{definition}

It is not easy to see right away why we set $\T_{\text{ABDF}}\left(t,\omega,(\theta,act)  \right)_2=alt_1$ 
precisely in these three cases:
this will become clear in the correspondence between a TASEP configuration $\eta$ and $\theta$.

The quantity $\T_{\text{ABDF}}\left(  t,\omega,\left(  \theta
,act\right)  \right)  _1(x)  $ can only take values in
$\left\{  -1,0,1\right\}  $. We shall prove that it satisfies
\autoref{def ABDF config preliminary}, (i)-(ii) and therefore $\T_{\text{ABDF}}\left(  t,\omega,\theta\right)  \in\Lambda$. 
To minimize double proofs, we postpone this fact to the verification of the link with TASEP 
(see \autoref{thm:conjugacy} below).

Once this is proved, one can introduce the ABDF random dynamical system,
of which $\T_{\text{ABDF}}\left(  t,\omega,\cdot\right)$ is just the 1-step dynamics at time $t$.
We let $\phi_{\text{ABDF}}\left(0,\omega\right)=id$, and for $t>0$, $t\in\N$,
\begin{equation*}
\phi_{\text{ABDF}}\left(  t,\omega\right):=\T_{\text{ABDF}%
}\left(  t,\omega\right)  \circ\T_{\text{ABDF}}\left(  t-1,\omega
\right)  \circ\cdot\cdot\cdot\circ\T_{\text{ABDF}}\left(
1,\omega\right),
\end{equation*}
so that it holds the random dynamical system property
\[
\phi_{\text{ABDF}}\left(  t,\omega\right)  \circ\phi_{\text{ABDF}}\left(
s,\omega\right)  =\phi_{\text{ABDF}}\left(  t+s,\omega\right), \quad
t,s\in\N,\, \omega\in\Omega.
\]

\section{TASEP, its pairs and ABDF\label{sect Tasep and ABDF}}

A TASEP configuration is a map
\[
\eta:\Z\rightarrow\left\{  0,1\right\}  .
\]
When $\eta(x)  =1$, we say that $x$ is occupied by a
particle; when $\eta(x)  =0$, we say that $x$ is empty.

TASEP dynamics in discrete time $t\in\N$ consists in
particles moving to the right by one position with probability $\frac12$,
with simultaneous independent jumps, 
aborted when the arrival position is occupied. More precisely, given a configuration $\eta$ at time $t-1\in
\N$, a particle at position $x\in\Z$ (which means $\eta
_{t-1}(x)  =1$) has probability $\frac12$ to jump on the
right at time $t$ (namely $\eta_{t}(x+1)  =1$), but the jump is
aborted if $\eta_{t-1}(x+1)  =1$.

Using the probability space $\left(  \Omega,\mathcal{F},P\right)  $ defined
above, when a particle is at time $t-1\in\N$ at positions
$x\in\Z$, it jumps if both $\omega\left(  t-1,x\right)  =1$ and the
position $x+1$ is free. Denote by $\T_{\text{TASEP}}\left(
t,\omega,\cdot\right)  $ the random map which associates to a given TASEP
configuration $\eta$ and a given random choice $\omega\in\Omega$ the
subsequent, one-time step, TASEP configuration. 
Heuristic prescriptions are summarized in
\[
\T_{\text{TASEP}}\left(  t,\omega,\eta\right)  (x)
=\left\{
\begin{array}
[c]{ccc}%
\eta(x)  & \text{if} & \eta(x)  =\eta\left(
x+1\right)  =1\\
\varkappa\left(  t-1,x\right)  & \text{if} & \eta(x)
=1,\eta(x+1)  =0\\
\omega\left(  t-1,x-1\right)  & \text{if} & \eta(x)
=0,\eta(x-1)  =1\\
\eta(x)  & \text{if} & \eta(x)  =\eta\left(
x-1\right)  =0
\end{array}
\right.,
\]
or equivalently
\begin{multline*}
	\T_{\text{TASEP}}\left(  t,\omega,\eta\right)  (x)\\
	=\left\{
	\begin{array}
	[c]{ccc}%
	\varkappa\left(  t-1,x\right)  \eta(x)  +\omega\left(
	t-1,x\right)  \eta(x+1)  & \text{if} & \eta(x)  =1\\
	\varkappa\left(  t-1,x-1\right)  \eta(x)  +\omega\left(
	t-1,x-1\right)  \eta(x-1)  & \text{if} & \eta(x)  =0
	\end{array}
	\right.,
\end{multline*}
which gives rise to the following rigorous Definition.

\begin{definition}
\label{def Tasep}Let $\left(  \Omega,\mathcal{F},P\right)  $ be the
probability space of \autoref{def ABDF dynamics}. We define the family
of maps $\T_{\text{TASEP}}\left(  t,\omega,\cdot\right)  $ on
$\left\{  0,1\right\}  ^\Z$, indexed by $t\in\N$ and
$\omega\in\Omega$, by
\begin{align}
\T_{\text{TASEP}}\left(  t,\omega,\eta\right)  (x)   &
=\left[  \varkappa\left(  t-1,x\right)  \eta(x)  +\omega\left(
t-1,x\right)  \eta(x+1)  \right]  \eta(x)
\label{Tasep map}\\
& +\left[  \varkappa\left(  t-1,x-1\right)  \eta(x)
+\omega\left(  t-1,x-1\right)  \eta(x-1)  \right]  \left(
1-\eta(x)  \right) \nonumber
\end{align}
when $t>0$, $\T_{\text{TASEP}}\left(  0,\omega,\cdot\right)  =id$.
\end{definition}
\noindent
As for ABDF above, we may introduce TASEP random dynamical system by setting
$\phi_{\text{TASEP}}\left(0,\omega\right)=id$ and, for $t>0$, $t\in\N$,
\begin{equation*}
\phi_{\text{TASEP}}\left(  t,\omega\right):=\T_{\text{TASEP}%
}\left(  t,\omega\right)  \circ\T_{\text{TASEP}}\left(
t-1,\omega\right)  \circ\cdot\cdot\cdot\circ\T_{\text{TASEP}}\left(
1,\omega\right),
\end{equation*}
the latter satisfying the random dynamical system property%
\[
\phi_{\text{TASEP}}\left(  t,\omega\right)  \circ\phi_{\text{TASEP}}\left(
s,\omega\right)  =\phi_{\text{TASEP}}\left(  t+s,\omega\right), \quad t,s\in\N, \, \omega\in\Omega.
\]
We now turn our attention to pairs in TASEP configurations: pairs of
particles and pairs of empty positions.

\begin{definition}
\label{def pair operator}The pair operator
\begin{align*}
 \PP:\left\{  0,1\right\}  ^\Z\rightarrow\left(  \left\{
-1,0,1\right\}  \times\left\{  0,1\right\}  \right)  ^\Z, \quad 
 \PP(\eta)  (x)=\left(   \PP
(\eta)  _1(x)  , \PP(\eta)
_2(x)  \right),  \quad x\in\Z,
\end{align*}
is the function defined as follows:

a) for every $\eta\in\left\{  0,1\right\}  ^\Z$ and $x\in\Z$,
\[
 \PP(\eta)  _1(x)  =1-\eta(x)-\eta(x+1),
\]

b) if $ \PP(\eta)  _1\neq0$, then $ \PP\left(
\eta\right)  _2=ar\left(   \PP(\eta)  _1\right)  $,
with $ar$ as in \autoref{def ABDF config preliminary},

c) if $ \PP(\eta)  _1=0$, namely $\eta=alt_{\alpha}$ for
$\alpha=0$ or $1$, then $ \PP(\eta)  _2=\eta$; in other
words, $ \PP\left(  alt_{\alpha}\right)  =\overline{alt}_{\alpha}$,
$\alpha=0,1$.
\end{definition}


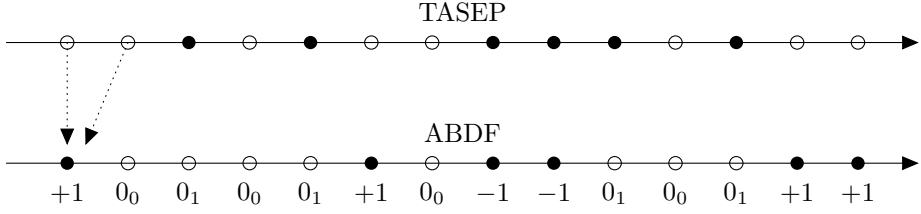
\begin{figure}
	\centering
	\begin{tikzpicture}[line cap=round,line join=round,>=triangle 45,x=0.8cm,y=0.8cm]
	\clip(1,0) rectangle (18,5);
	\draw [->] (1,1) -- (16,1);
	\draw [->] (1,3) -- (16,3);
	\draw [->,dotted] (2,3) -- (2,1.3);
	\draw [->,dotted] (3,3) -- (2.3,1.3);
	\draw [color=black] (2,3) circle (2.5pt);
	\draw [color=black] (3,3) circle (2.5pt);
	\fill [color=black] (4,3) circle (2.5pt);
	\draw [color=black] (5,3) circle (2.5pt);
	\fill [color=black] (6,3) circle (2.5pt);
	\draw [color=black] (7,3) circle (2.5pt);
	\draw [color=black] (8,3) circle (2.5pt);
	\fill [color=black] (9,3) circle (2.5pt);
	\fill [color=black] (10,3) circle (2.5pt);
	\fill [color=black] (11,3) circle (2.5pt);
	\draw [color=black] (12,3) circle (2.5pt);
	\fill [color=black] (13,3) circle (2.5pt);
	\draw [color=black] (14,3) circle (2.5pt);
	\draw [color=black] (15,3) circle (2.5pt);
	\fill [color=black] (2,1) circle (2.5pt);
	\draw [color=black] (3,1) circle (2.5pt);
	\draw [color=black] (4,1) circle (2.5pt);
	\draw [color=black] (5,1) circle (2.5pt);
	\draw [color=black] (6,1) circle (2.5pt);
	\fill [color=black] (7,1) circle (2.5pt);
	\draw [color=black] (8,1) circle (2.5pt);
	\fill [color=black] (9,1) circle (2.5pt);
	\fill [color=black] (10,1) circle (2.5pt);
	\draw [color=black] (11,1) circle (2.5pt);
	\draw [color=black] (12,1) circle (2.5pt);
	\draw [color=black] (13,1) circle (2.5pt);
	\fill [color=black] (14,1) circle (2.5pt);
	\fill [color=black] (15,1) circle (2.5pt);
	\draw (2,0.5) node {$+1$};
	\draw (3,0.5) node {$0_0$};
	\draw (4,0.5) node {$0_1$};
	\draw (5,0.5) node {$0_0$};
	\draw (6,0.5) node {$0_1$};
	\draw (7,0.5) node {$+1$};
	\draw (8,0.5) node {$0_0$};
	\draw (9,0.5) node {$-1$};
	\draw (10,0.5) node {$-1$};
	\draw (11,0.5) node {$0_1$};
	\draw (12,0.5) node {$0_0$};
	\draw (13,0.5) node {$0_1$};
	\draw (14,0.5) node {$+1$};
	\draw (15,0.5) node {$+1$};
	\draw (8.5,1.5) node {ABDF};
	\draw (8.5,3.5) node {TASEP};
	\end{tikzpicture}
	\caption{The TASEP configuration associated with the ABDF one of \autoref{fig:abdfconfig}.
		Dotted arrows show (for the right-most site) two TASEP site determining
		the state of a ABDF one.}
	\label{fig:abdffromtasep} 
\end{figure}

The link between TASEP pairs and ABDF configurations is the following conjugation result
between the random dynamical systems $\phi_{\text{ABDF}}\left(
t,\omega\right)  $ and $\phi_{\text{TASEP}}\left(  t,\omega\right)  $.

\begin{theorem}\label{thm:conjugacy}
	\begin{itemize}
		\item[a)] The pair map $ \PP$ is a bijection between $\left\{0,1\right\}  ^\Z$ and $\Lambda$;
		\item[b)] For every $\eta\in\left\{  0,1\right\}  ^\Z$,
		$t\in\N$, $\omega\in\Omega$,
		\[
		\T_{\text{ABDF}}\left(  t,\omega, \PP(\eta)
		\right)  = \PP\left(  \T_{\text{TASEP}}\left(  t,\omega
		,\eta\right)  \right),
		\]
		in particular, $\T_{\text{ABDF}}\left(  t,\omega,\Lambda\right)\subset\Lambda$.
	\end{itemize}
\end{theorem}

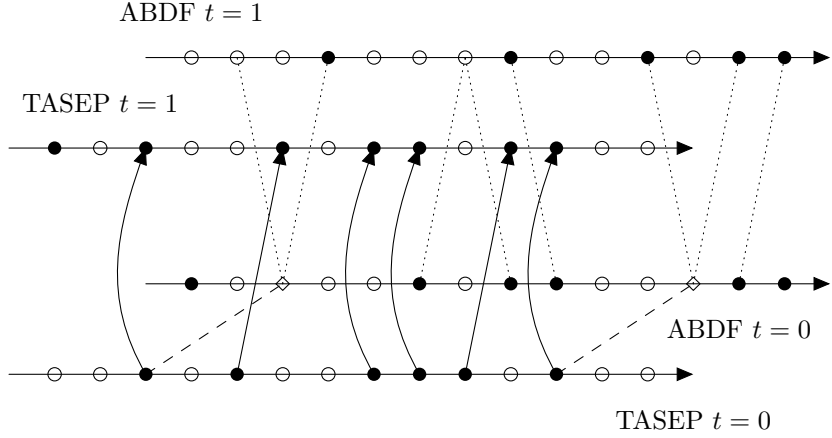
\begin{figure}
	\centering
	\begin{tikzpicture}[line cap=round,line join=round,>=triangle 45,x=0.6cm,y=0.6cm]
	\clip(0,-1) rectangle (23,11);
	\draw [->] (1,1) -- (16,1);
	\draw [->] (4,3) -- (19,3);
	\draw [->] (1,6) -- (16,6);
	\draw [->] (4,8) -- (19,8);
	\draw [dash pattern=on 4pt off 4pt] (13,1)-- (16,3);
	\draw [dash pattern=on 4pt off 4pt] (4,1)-- (7,3);
	\draw [dotted] (7,3)-- (6,8);
	\draw [dotted] (7,3)-- (8,8);
	\draw [dotted] (10,3)-- (11,8);
	\draw [dotted] (12,3)-- (11,8);
	\draw [dotted] (13,3)-- (12,8);
	\draw [dotted] (16,3)-- (15,8);
	\draw [dotted] (16,3)-- (17,8);
	\draw [dotted] (17,3)-- (18,8);
	\draw [->] (4,1) to [out=120,in=250] (4,6);
	\draw [->] (6,1) -- (7,6);
	\draw [->] (9,1) to [out=120,in=250] (9,6);
	\draw [->] (10,1) to [out=120,in=250] (10,6);
	\draw [->] (11,1) -- (12,6);
	\draw [->] (13,1) to [out=120,in=250] (13,6);
	\draw [color=black] (2,1) circle (2.5pt);
	\draw [color=black] (3,1) circle (2.5pt);
	\fill [color=black] (4,1) circle (2.5pt);
	\draw [color=black] (5,1) circle (2.5pt);
	\fill [color=black] (6,1) circle (2.5pt);
	\draw [color=black] (7,1) circle (2.5pt);
	\draw [color=black] (8,1) circle (2.5pt);
	\fill [color=black] (9,1) circle (2.5pt);
	\fill [color=black] (10,1) circle (2.5pt);
	\fill [color=black] (11,1) circle (2.5pt);
	\draw [color=black] (12,1) circle (2.5pt);
	\fill [color=black] (13,1) circle (2.5pt);
	\draw [color=black] (14,1) circle (2.5pt);
	\draw [color=black] (15,1) circle (2.5pt);
	\fill [color=black] (2,6) circle (2.5pt);
	\draw [color=black] (3,6) circle (2.5pt);
	\fill [color=black] (4,6) circle (2.5pt);
	\draw [color=black] (5,6) circle (2.5pt);
	\draw [color=black] (6,6) circle (2.5pt);
	\fill [color=black] (7,6) circle (2.5pt);
	\draw [color=black] (8,6) circle (2.5pt);
	\fill [color=black] (9,6) circle (2.5pt);
	\fill [color=black] (10,6) circle (2.5pt);
	\draw [color=black] (11,6) circle (2.5pt);
	\fill [color=black] (12,6) circle (2.5pt);
	\fill [color=black] (13,6) circle (2.5pt);
	\draw [color=black] (14,6) circle (2.5pt);
	\draw [color=black] (15,6) circle (2.5pt);
	\fill [color=black] (5,3) circle (2.5pt);
	\draw [color=black] (6,3) circle (2.5pt);
	\draw [color=black] (7,3) ++(-2.5pt,0 pt) -- ++(2.5pt,2.5pt)--++(2.5pt,-2.5pt)--++(-2.5pt,-2.5pt)--++(-2.5pt,2.5pt);
	\draw [color=black] (8,3) circle (2.5pt);
	\draw [color=black] (9,3) circle (2.5pt);
	\fill [color=black] (10,3) circle (2.5pt);
	\draw [color=black] (11,3) circle (2.5pt);
	\fill [color=black] (12,3) circle (2.5pt);
	\fill [color=black] (13,3) circle (2.5pt);
	\draw [color=black] (14,3) circle (2.5pt);
	\draw [color=black] (15,3) circle (2.5pt);
	\draw [color=black] (16,3) ++(-2.5pt,0 pt) -- ++(2.5pt,2.5pt)--++(2.5pt,-2.5pt)--++(-2.5pt,-2.5pt)--++(-2.5pt,2.5pt);
	\fill [color=black] (17,3) circle (2.5pt);
	\fill [color=black] (18,3) circle (2.5pt);
	\draw [color=black] (5,8) circle (2.5pt);
	\draw [color=black] (6,8) circle (2.5pt);
	\draw [color=black] (7,8) circle (2.5pt);
	\fill [color=black] (8,8) circle (2.5pt);
	\draw [color=black] (9,8) circle (2.5pt);
	\draw [color=black] (10,8) circle (2.5pt);
	\draw [color=black] (11,8) circle (2.5pt);
	\fill [color=black] (12,8) circle (2.5pt);
	\draw [color=black] (13,8) circle (2.5pt);
	\draw [color=black] (14,8) circle (2.5pt);
	\fill [color=black] (15,8) circle (2.5pt);
	\draw [color=black] (16,8) circle (2.5pt);
	\fill [color=black] (17,8) circle (2.5pt);
	\fill [color=black] (18,8) circle (2.5pt);
	\draw (16,0) node {TASEP $t=0$};
	\draw (17,2) node {ABDF $t=0$};
	\draw (3,7) node {TASEP $t=1$};
	\draw (5,9) node {ABDF $t=1$};
	\end{tikzpicture}
	\caption{TASEP and ABDF evolutions, starting from \autoref{fig:abdffromtasep}.
		Solid arrows denote trajectories of TASEP particles, dotted lines the ones of ABDF as above.
		Two dashed lines couple generations of ABDF particles to TASEP particles not jumping even if they can do so.}
	\label{fig:abdfandtasep} 
\end{figure}

\subsection{Proof of \autoref{thm:conjugacy}, a).}\label{subsect proof bijection}
The proof in itself can be made more concise than what follows,
but we take the chance to introduce some more structure.
We begin with a simple observation, that can be used to invert ``by hand'' the pair map $ \PP$.

\begin{remark}\label{rem inversion}
	Given $\eta\in\set{0,1}^\Z$, let $\theta(x)  =1-\eta(x)-\eta(x+1) $, $x\in\Z$. Then, for every $n\in\N$
	\[
	\theta(x+n)  =1-\eta(x+n)  -\eta\left(x+n+1\right),
	\]
	hence
	\begin{align*}
	\eta\left(  x+n+1\right)   & =1-\eta(x+n)  -\theta\left(
	x+n\right) \\
	& =\eta\left(  x+n-1\right)  +\theta\left(  x+n-1\right)  -\theta\left(
	x+n\right) \\
	& =1-\eta\left(  x+n-2\right)  -\theta\left(  x+n-2\right)  +\theta\left(
	x+n-1\right)  -\theta(x+n),
	\end{align*}
	and so on, which gives us
	\begin{equation}\label{inversion formula}
	\eta\left(  x+n+1\right)  =\frac{1+(-1)  ^{n}}{2}-\left(
	-1\right)  ^{n}\eta(x)  -\sum_{k=0}^{n}(-1)
	^{k}\theta\left(  x+n-k\right)  .
	\end{equation}
	A similar formula holds for negative integer $n$. Hence, we may reconstruct
	$\eta$ from $\theta$ at the price of fixing one value of $\eta$, say $\eta(0)$.
\end{remark}

With this ``reconstruction algorithm'' at hand, we can proceed with the proof.
A crucial property is that a right pair and a left pair are always separated by an odd
number (= 1,3,5,...) of empty pairs, two right pairs or two left pairs by an even number
of empty pairs (= 0,2,4,...), see \autoref{fig:abdffromtasep}.

For $x_1,x_2\in\Z$ we set $[x_1,x_2]_\Z=(x_1,x_1+1,\dots x_2-1,x_2)$

\begin{definition}\label{def max altern interv}
Given a segment $\left[  x_1,x_2\right]_\Z$, we define
\begin{gather*}
	\PP_{\left[  x_1,x_2\right]  _\Z}:\left\{  0,1\right\}
	^{\left[  x_1,x_2+1\right]  _\Z}\rightarrow\left\{
	-1,0,1\right\}  ^{\left[  x_1,x_2\right]  _\Z}\\
	\PP_{\left[  x_1,x_2\right]  _\Z}(\eta)
	(x)  =1-\eta(x)  -\eta(x+1),
	\qquad \eta\in\left\{  0,1\right\}  ^{\left[  x_1,x_2+1\right]_\Z},\, x\in\left[  x_1,x_2\right]_\Z.
\end{gather*}
A segment $\left[  x_1,x_2\right]
_\Z$ of cardinality $n+2$ is called a maximal alternating segment of
$\eta\in\left\{  0,1\right\}^\Z$ if:
\begin{itemize}
	\item $\eta(x_1)=\eta(x_1+1)$ and $\eta(x_2)=\eta(x_2+1)$, \emph{i.e.} $ \PP(\eta)_1(x_1), \PP(\eta)_1(x_2)\in\set{\pm1},$	
	\item $\eta\left(  x_1+k\right)  \neq\eta\left(  x_1+k+1\right)  $, \emph{i.e.}	$ \PP(\eta)_1(  x_1+k) =0$, for
	$k=1,2,...,n$ (not imposed if $n=0$).
\end{itemize}
It is called a maximal alternating segment of $\theta\in\left\{
-1,0,1\right\}  ^\Z$ if
\begin{itemize}
	\item $\theta(x_1)  ,\theta(x_2)  \in\left\{-1,1\right\},$
	\item $\theta\left(  x_1+k\right)  =0$ for $k=1,2,...,n$ (not imposed when $n=0$).
\end{itemize}
Maximal alternating half lines
$(-\infty,x_2]_\Z$ and $[x_1,\infty)_\Z$ are defined analogously.
\end{definition}

Clearly, if $\left[  x_1,x_2\right]  _\Z$ is a maximal
alternating segment of $\eta\in\left\{  0,1\right\}  ^\Z$, then it
is a maximal alternating segment of $\theta:= \PP(\eta)
_1$ (similarly for half lines). The two key facts on this concept are
expressed by the following two lemmata.

\begin{lemma}\label{lem:parity}
Let $\left[  x_1,x_2\right]_\Z$ be a maximal alternating
segment of $\eta\in\left\{  0,1\right\}  ^\Z$ (thus of $\theta:= \PP(\eta)  _1$)
with cardinality $n+2$. Then:
\begin{itemize}
\item if $\theta(x_1)  \theta(x_2)  =1$, $n$ is even;
\item if $\theta(x_1)  \theta(x_2)  =-1$, $n$ is odd.
\end{itemize}
\end{lemma}

\begin{proof}
If $\theta(x_1)  \theta(x_2)  =1$, then $\theta(x_1)=\theta(x_2)=\pm 1$. 
In the $+1$ case, $\eta(x_1)  =\eta\left(  x_1+1\right)  =0$
, $\eta(x_2)  =\eta\left(  x_2+1\right)  =0$. Then, from
identity \eqref{inversion formula},%
\begin{align*}
\eta\left(  x_1+n+1\right)   & =\frac{1+(-1)  ^{n}}{2}-\left(
-1\right)  ^{n}\eta(x_1)  -\sum_{k=0}^{n}(-1)
^{k}\theta\left(  x_1+n-k\right) \\
& =\frac{1+(-1)  ^{n}}{2}-(-1)  ^{n}\theta\left(
x_1\right)=\frac{1+(-1)  ^{n}}{2}-(-1)  ^{n},
\end{align*}
where we have used $\theta\left(  x_1+n-k\right)  =0$ for $k=0,1,...,n-1$,
in the second-last step.
Since $\eta\left(  x_1+n+1\right)  =0$, this implies $n$ even. Other cases are analogous.
\end{proof}
\noindent
\autoref{lem:parity} imposes a restriction to the sequences of $\left\{-1,0,1\right\}^\Z$ 
in the range of the first component of $ \PP$.

\begin{lemma}
Let $\theta\in\Lambda\backslash\left\{  \overline{alt}_{0},\overline{alt}%
_1\right\}  $ and let $\left[  x_1,x_2\right]  _\Z$ be a
maximal alternating segment of $\theta$ of cardinality $n+2$. Then there
exists a unique $\eta\in\left\{  0,1\right\}  ^{\left[  x_1,x_2+1\right]  _\Z}$ such that
\[
 \PP_{\left[  x_1,x_2\right]  _\Z}(\eta)
=\theta|_{\left[  x_1,x_2\right]  _\Z}.
\]
The string $\eta$ satisfies $\eta(x_1)  =\eta\left(
x_1+1\right)  $, $\eta(x_2)  =\eta\left(  x_2+1\right)  $,
with the unique values determined by the values $\theta(x_1)
,\theta(x_2)  $ and in the middle it is given by
\eqref{inversion formula}, precisely
\[
\eta\left(  x_1+j+1\right)  =\frac{1+(-1)  ^{j}}{2}-\left(
-1\right)  ^{j}\eta(x_1)  -\sum_{k=0}^{j}(-1)
^{k}\theta\left(  x_1+j-k\right)
\]
for $j=1,...,n-1$.
\end{lemma}

\begin{proof}
If $\eta\in\left\{  0,1\right\}  ^{\left[  x_1,x_2+1\right]  _{\Z%
}} $ satisfies $ \PP_{\left[  x_1,x_2\right]  _\Z}\left(
\eta\right)  =\theta|_{\left[  x_1,x_2\right]  _\Z}$, then the
properties of the values of $\eta$, including the formula for the intermediate
values, are obvious or have been proved above. Thus uniqueness is clear.
Proving the existence means proving that $\eta\left(  x_1+n+1\right)  $
given by the formula coincides with the values of $\eta(x_2)
=\eta\left(  x_2+1\right)  $ prescribed by $\theta(x_2)  $.
This must be checked case by case, and we only report $\theta(x_1)  =\theta(x_2)=1$:
$n$ is even since $\theta\in\Lambda$, and
$\eta(x_1)  =\eta\left(  x_1+1\right)  $, $\eta\left(
x_2\right)  =\eta\left(  x_2+1\right)  $, all equal to zero. We have from \eqref{inversion formula}
\begin{align*}
\eta\left(  x_1+n+1\right)   & =\frac{1+(-1)  ^{n}}{2}-\left(
-1\right)  ^{n}\eta(x_1)  -\sum_{k=0}^{n}(-1)
^{k}\theta\left(  x_1+n-k\right) \\
& =\frac{1+(-1)  ^{n}}{2}-(-1)  ^{n}=0=\eta\left(
x_2\right).\qedhere
\end{align*}
\end{proof}

\begin{proof}[Proof of \autoref{thm:conjugacy}, a)]
It is sufficient to prove that, given $\theta\in\Lambda\backslash\left\{
\overline{alt}_{0},\overline{alt}_1\right\}  $, there exists one and only
one $\eta\in\left\{  0,1\right\}  ^\Z\backslash\left\{
alt_0,alt_1\right\}  $ such that $ \PP(\eta)  =\theta$.

Let $\left\{  x_{n}\right\}  $ be the strictly increasing, possibly bi-infinite
sequence of points of $\Z$ such that $[x_{n},x_{n+1}]_\Z$ is an even or odd maximal segment of $\theta$. 
There are four cases: $\left\{  x_{n}\right\}  $ is bi-infinite, or infinite only to the
left, or infinite only to the right, or finite. The construction of $\left\{
x_{n}\right\}  $ may proceed from the origin of $\Z$: we denote by
$x_0$ the first point $\geq0$ with $\theta\left(  x_0\right)  \neq0$; by
$x_1$ the first point $>x_0$ such that $\theta(x_1)  \neq0
$; and so on, obviously if they exist. And we denote by $x_{-1}$ the first
point $<0$ such that $\theta\left(  x_{-1}\right)  \neq0$ and so on.

For every $n$ such that $x_{n},x_{n+1}$ exists, we construct the corresponding
values of $\eta\left(  x_{n}\right)  $, ..., $\eta\left(  x_{n+1}+1\right)  $
as in (d) of the previous lemma. The construction is unique with the property
that, locally, $ \PP(\eta)  _1=\theta$ on $\left[
x_{n},x_{n+1}\right]  _\Z$. However, in principle the definition for
$\left[  x_{n},x_{n+1}\right]  _\Z$ may be in contradiction with the
definition for $\left[  x_{n+1},x_{n+2}\right]  _\Z$ because the
points $x_{n+1},x_{n+1}+1$ are in common. But, based on $\left[  x_{n}%
,x_{n+1}\right]  _\Z$, we have defined $\eta\left(  x_{n+1}\right)
=\eta\left(  x_{n+1}+1\right)  $, equal to 1 if $\theta\left(  x_{n+1}\right)
=-1$, equal to 0 if $\theta\left(  x_{n+1}\right)  =1$. And based on $\left[
x_{n+1},x_{n+2}\right]  _\Z$ we have given the same definition.
Therefore there is no contradiction. The treatment of half lines is analogous.
\end{proof}

\subsection{Proof of \autoref{thm:conjugacy}\label{subsect proof conjug}}

We already stressed the drawback of $ar\left(  \theta\right)  $ being non local, but when $\theta= \PP(\eta)  $, 
the expression of $ar\left(  \theta\right)  (x)  $ becomes local when written
in terms of $\eta$ (which depends non-locally on $\theta$). This is a key fact
for the proof of \autoref{thm:conjugacy}.

\begin{lemma}
\label{lemma key ident}If $(\theta,act)  = \PP\left(
\eta\right)  $, then%
\[
act(x)  =\eta(x)  \left(  1-\eta(x+1)
\right)  .
\]
In other words, $act(x)  $ is equal to one if and only if
$\eta(x)  =1$, $\eta(x+1)  =0$, namely there is a
particle at $x$ and the position $x+1$ is free, so the particle can jump.
\end{lemma}

\begin{proof}
Let us treat separately the case when $\eta=alt_{\alpha}$, $\alpha=0,1$. In
this case $act=$ $\eta$; and also $\eta(x)  \left(  1-\eta\left(
x+1\right)  \right)  =\eta(x)  $, because if $\eta\left(
x+1\right)  =0$ it is true, while if $\eta(x+1)  =1$ we
necessarily have $\eta(x)  =0$ by alternation, which coincides
with $\eta(x)  \left(  1-\eta(x+1)  \right)  $. The
formula of the lemma is proved in this particular case.

Assume now $\eta$ different from $alt_{\alpha}$, $\alpha=0,1$, so that
$act=ar\left(  \theta\right)  $. Recall the definition of $ar\left(
\theta\right)  (x)  $ in \autoref{def ABDF config}, points
(iii)-(vi). Let $x_0$ be such that $\theta\left(  x_0\right)  \neq0$. It
means that $\eta\left(  x_0\right)  =\eta\left(  x_0+1\right)  $, both
equal to 0 or 1. In both cases $\eta\left(  x_0\right)  \left(
1-\eta\left(  x_0+1\right)  \right)  =0$, hence equal to $ar\left(
\theta\right)  \left(  x_0\right)  $ as defined in Definition
\ref{def ABDF config} point (iii).

Assume now $\theta\left(  x_0\right)  =0$, from which $\eta\left(x_0\right)\neq\eta\left(  x_0+1\right)$ 
and thus the pair $\left(\eta\left(  x_0\right),\eta\left(  x_0+1\right)  \right)  $ is either
$\left(  1,0\right)  $ or $\left(  0,1\right)  $. Assume that the set
$L\left(  x_0\right)  $ is non empty and let $x_1$ be its maximum and let
$k>0$ be such that $x_1+k=x_0$. The proof can be divided into several
cases depending on the value of $\theta(x_1)  $ and the parity
of $k$. For instance, assume $\theta(x_1)  =1$, $k$ odd. Thus
$\eta(x_1)  =\eta\left(  x_1+1\right)  =0$, $\eta\left(
x_1+2\right)  =1$, $\eta\left(  x_1+3\right)  =0$, and so on, hence
$\eta\left(  x_0\right)  =\eta\left(  x_1+k\right)  =0$, and $\eta\left(
x_0+1\right)  =1$. In this case
\[
\eta\left(  x_0\right)  \left(  1-\eta\left(  x_0+1\right)  \right)  =0
\]
and (from \autoref{def ABDF config} point (iv))%
\[
ar\left(  \theta\right)  \left(  x_0\right)  =\left\vert \frac{\theta\left(
x_1\right)  +(-1)  ^{k}}{2}\right\vert =0
\]
so they coincide. If $\theta(x_1)  =1$, $k$ even,
\begin{equation*}
	\eta\left(  x_0\right)  \left(  1-\eta\left(  x_0+1\right)  \right)=1, \quad
	ar\left(  \theta\right)  \left(  x_0\right) =\left\vert \frac
	{\theta(x_1)  +(-1)  ^{k}}{2}\right\vert =1,
\end{equation*}
so they coincide. The reader can check the two cases with $\theta\left(
x_1\right)  =0$. If $L\left(  x_0\right)  $ is empty and $R\left(
x_0\right)  $ is non empty, the arguments are similar.
\end{proof}

\begin{proof}
[Proof of \autoref{thm:conjugacy}]\textbf{Step 1}. We prove the identity between the first components:
\begin{equation}
\T_{\text{ABDF}}\left(  t,\omega, \PP(\eta)
\right)  _1= \PP\left(  \T_{\text{TASEP}}\left(
t,\omega,\eta\right)  \right)  _1.\label{identity first component stoch}%
\end{equation}
Let $\eta\in\left\{  0,1\right\}  ^\Z$, $t\in\N$, $\omega
\in\Omega$, be given and write $\theta:= \PP(\eta)  _1$,
$\widehat{\eta}:=\T_{\text{TASEP}}\left(  t,\omega,\eta\right)  $,
$\widehat{\theta}:= \PP\left(  \widehat{\eta}\right)  _1$,
$\widetilde{\theta}:=\T_{\text{ABDF}}\left(  t,\omega,\theta\right)
_1$. We have to prove $\widehat{\theta}=\widetilde{\theta}$.

From \autoref{def pair operator} and \autoref{def Tasep},
\begin{align*}
\widehat{\theta}(x)   
& =1-\widehat{\eta}(x)-\widehat{\eta}(x+1) 
=1-\left[  \varkappa\left(t,x\right)  \eta(x)  +\omega\left(
t,x\right)  \eta(x+1)  \right]  \eta(x) \\
&\qquad -\left[  \varkappa(t,x-1)  \eta(x)  +\omega\left(
t,x-1\right)  \eta(x-1)  \right]  \left(  1-\eta(x)
\right) \\
&\qquad -\left[  \varkappa(t,x+1)  \eta(x+1)
+\omega(t,x+1)  \eta(t,x+2)  \right]  \eta\left(
x+1\right) \\
&\qquad -\left[  \varkappa\left(  t,x\right)  \eta(x+1)  +\omega\left(
t,x\right)  \eta(x)  \right]  \left(  1-\eta(x+1)
\right)  .
\end{align*}
It simplifies, for instance, to
\begin{align*}
\widehat{\theta}(x)   & =\left[  1-\omega(t,x-1)
\eta(x-1)  \right]  (1-\eta(x)) \\
& -\left[  \varkappa(t,x+1)  +\omega(t,x+1)
\eta(t,x+2)  \right]  \eta(x+1),
\end{align*}
because $\eta(x)  \eta(x)  =\eta(x)  $,
$\eta(x)  (1-\eta(x))  =0$,
$\varkappa\left(  t,x\right)  +\omega\left(  t,x\right)  =1$.

Concerning $\widetilde{\theta}$, from Definitions \ref{def ABDF dynamics} and
\ref{def pair operator},%
\begin{align*}
\widetilde{\theta}(x)   & =\max\left(  \theta(x-1)
,0\right)  +act(x-1)  \varkappa(t,x-1) \\
& +\min\left(  \theta(x+1)  ,0\right)  -act(x+1)
\varkappa(t,x+1)  ,
\end{align*}
where, by definition of $ \PP(\eta)  _1$ and from Lemma
\ref{lemma key ident},
\begin{equation*}
\theta(x)   =1-\eta(x)  -\eta(x+1), \quad 
act(x)  =\eta(x)  \left(  1-\eta\left(x+1\right)  \right)  .
\end{equation*}
We have,
\begin{gather*}
	\max\left(  1-\eta(x-1)  -\eta(x)  ,0\right)
	=\left(  1-\eta(x-1)  \right)  \left(  1-\eta(x)
	\right),\\
	\min\left(  1-\eta(x+1)  -\eta(t,x+2)  ,0\right)
	=-\eta(x+1)  \eta(t,x+2)  .
\end{gather*}
Moreover,
\begin{equation*}
act(x-1)   =\eta(x-1)  \left(  1-\eta\left(
x\right)  \right), \quad
act(x+1)   =\eta(x+1)  \left(  1-\eta\left(
x+2\right)  \right)  .
\end{equation*}
Hence
\begin{align*}
\widetilde{\theta}(x)   & =\left(  1-\eta(x-1)\right)  (1-\eta(x))  +\eta(x-1)
(1-\eta(x))  \varkappa(t,x-1) \\
&\quad  -\eta(x+1)  \eta(t,x+2)  -\eta(x+1)
\left(  1-\eta(t,x+2)  \right)  \varkappa(t,x+1)\\
& =1-\eta(x)  -\eta\left(x-1\right)  (1-\eta(x))  \omega\left(t,x-1\right) \\
&\quad -\eta(x+1)  \varkappa(t,x+1)
-\eta(x+1)  \eta(t,x+2)  \omega(t,x+1)
\end{align*}
which is equal to $\widehat{\theta}(x)  $.

\textbf{Step 2}. We now prove the identity between second components:
\begin{equation}
\T_{\text{ABDF}}\left(  t,\omega, \PP(\eta)
\right)  _2= \PP\left(  \T_{\text{TASEP}}\left(
t,\omega,\eta\right)  \right)  _2.\label{identity second component stoch}%
\end{equation}
This identity is obviously true when the two elements of
\eqref{identity first component stoch} are not zero, because both the terms of
\eqref{identity second component stoch} are defined as the activation map of
the corresponding terms of \eqref{identity first component stoch}. 
Thus it remains to prove identity \eqref{identity second component stoch} when
\begin{equation*}
\T_{\text{ABDF}}\left(  t,\omega, \PP(\eta)
\right)  _1  =0, \quad  \PP\left(  \T_{\text{TASEP}}\left(  t,\omega,\eta\right)
\right)  _1  =0.
\end{equation*}
Condition $ \PP\left(  \T_{\text{TASEP}}\left(t,\omega,\eta\right)  \right)_1=0$ implies
$\T_{\text{TASEP}}\left(  t,\omega,\eta\right)  =alt_{\alpha}$ for some
$\alpha=0,1$ and $ \PP\left(  \T_{\text{TASEP}}\left(t,\omega,\eta\right)  \right)  _2=alt_{\alpha}$. 
Therefore we have to prove
\[
\T_{\text{ABDF}}\left(  t,\omega, \PP(\eta)\right)_2=alt_{\alpha},
\]
and we split the proof in two more steps.

\textbf{Step 3}. We continue the proof of Step 2 assuming $\alpha=1$. We have
to prove that one of the following three conditions hold (see the three
conditions at the end of \autoref{def ABDF dynamics}):
\begin{gather}\nonumber
	\PP(\eta)  _1(0)  = -1;\\ \label{three conditions}
	\PP(\eta)  _2(-1)  = 1 \quad \text{and} \quad
	\omega\left(  t-1,-1\right)  = 1;\\ \nonumber
	\PP(\eta)  _2(0)  = 1 \quad \text{and}\quad 
	\omega\left(  t-1,0\right)  = 0.
\end{gather}
If the first one is true, the proof is complete. Otherwise we have
$ \PP(\eta)  _1(0)  =1$ or $ \PP(\eta)  _1(0)  =0$. Let us prove that
$ \PP(\eta)  _1(0)  =1$ implies the second condition; 
and that $ \PP(\eta)_1(0)  =0$ implies either the second or third conditions.

Thus assume $ \PP(\eta)  _1(0)  =1$. In
this case $\eta(0)  =\eta\left(  1\right)  =0$, hence we need
$\eta(-1)  =1$ and $\omega\left(  t-1,-1\right)  =1$ to get
$\T_{\text{TASEP}}\left(  t,\omega,\eta\right)  =alt_1$. But then,
from $\eta(-1)  =1$, $\eta(0)  =0$ and
$\omega\left(  t-1,-1\right)  =1$ we deduce $ \PP(\eta)
_1(-1)  =0$ and $ \PP(\eta)  _2\left(
-1\right)  =1$ (\autoref{lemma key ident}), so the second condition hold true.

If $ \PP(\eta)  _1(0)  =0$, then
$\eta(0)  \neq\eta\left(  1\right)  $. It cannot be $\eta\left(
0\right)  =1$, $\eta\left(  1\right)  =0$, $\omega\left(  t-1,0\right)  =1$,
otherwise $\T_{\text{TASEP}}\left(  t,\omega,\eta\right)  \left(
1\right)  =1$, incompatible with $\T_{\text{TASEP}}\left(
t,\omega,\eta\right)  =alt_1$. Thus: i) either $\eta(0)  =0$,
$\eta\left(  1\right)  =1$; ii) or $\eta(0)  =1$, $\eta\left(
1\right)  =0$, $\omega\left(  t-1,0\right)  =0$. In case (i), we must have
$\eta(-1)  =1$ and $\omega\left(  t-1,-1\right)  =1$ to get
$\T_{\text{TASEP}}\left(  t,\omega,\eta\right)  =alt_1$; in this
case the conclusion is, as above, that the second condition holds true. In
case (ii) we have $ \PP(\eta)  _1(0)  =0$,
$ \PP(\eta)  _2(0)  =1$ (\autoref{lemma key ident}) and of course $\omega\left(  t-1,0\right)  =0$, hence
the last of the three conditions hold. The case $\alpha=1$ is solved.

\textbf{Step 4}. We continue the proof of Step 2 assuming $\alpha=0$,
namely $\T_{\text{TASEP}}\left(  t,\omega,\eta\right)  =alt_0$. We
have to prove that none of the conditions \eqref{three conditions} hold. 
We argue by contradiction, observing that
\begin{align*}
	\PP(\eta)_1(0)=-1 \, 
	&\Rightarrow\, \eta(0)  =\eta\left(  1\right)=1 \,
	\Rightarrow\, \T_{\text{TASEP}}\left(  t,\omega,\eta\right)(0)  =1,\\
	\begin{cases}
	\PP(\eta)_2(-1)  =1\\
	\omega\left(  t-1,-1\right)  =1
	\end{cases} \,
	&\Rightarrow\, \eta(-1)  =1, \eta(0)  =0 \,
	\Rightarrow\, \T_{\text{TASEP}}\left(  t,\omega,\eta\right)  (0)  =1,\\
	\begin{cases}
	\PP(\eta)  _2(0)=1\\
	\omega\left(  t-1,0\right)  =0
	\end{cases}
	\, 
	&\Rightarrow\,  \eta(0)  =1, \eta\left(  1\right)  =0\,
	\Rightarrow\,
	\begin{cases}
	\T_{\text{TASEP}}\left(	t,\omega,\eta\right)  (0)  =1\\
	\T_{\text{TASEP}}\left(  t,\omega,\eta\right)  \left(  1\right)  =0
	\end{cases},
\end{align*}
where conditions on the right are incompatible with $\T_{\text{TASEP}}\left(t,\omega,\eta\right)=alt_0$.
This completes the proof.
\end{proof}

We complete this section with a simple corollary of \autoref{thm:conjugacy} 
which is not easy to prove directly on ABDF dynamics. In
plain words it says that the point $x_0$ where a pair coalesces, cannot be the
origin of a pair at the same time of the coalescence.

\begin{corollary}
\label{corollary coalescence arising}Let $x_0\in\Z$, $\left(
\theta,act\right)  \in\Lambda$ be such that
\[
\theta\left(  x_0-1\right)  =1,\theta\left(  x_0\right)  =0,\theta\left(
x_0+1\right)  =-1.
\]
Then, for every $\left(  t,\omega\right)  \in\N\times\Omega$
\begin{equation*}
\T_{\text{ABDF}}\left(  t,\omega,(\theta,act)  \right)
_1\left(  x_0\right) =0, \quad 
\T_{\text{ABDF}}\left(  t,\omega,(\theta,act)  \right)
_2\left(  x_0\right)   =0.
\end{equation*}
The same result holds if one or both $x_0-1,x_0+1$ are arising pair point
for $\theta$.
\end{corollary}

\begin{proof}
Let $\eta$ be such that $(\theta,act)  = \PP\left(
\eta\right)  $. By hypothesis
\[
\eta\left(  x_0-1\right)  =0,\eta\left(  x_0\right)  =0,\eta\left(
x_0+1\right)  =1,\eta\left(  x_0+2\right)  =1.
\]
TASEP dynamics cannot change the values at $x_0$ and $x_0+1$, hence%
\[
\T_{\text{TASEP}}\left(  t,\omega,\eta\right)  \left(  x_0\right)
=0,\T_{\text{TASEP}}\left(  t,\omega,\eta\right)  \left(
x_0+1\right)  =1.
\]
This implies the result in the first case.

Now, assume $x_0-1$ is an arising pair point for $\left(  \left(
\theta,act\right)  ,\omega\right)  $, and $\theta\left(  x_0\right)  =0$,
$\theta\left(  x_0+1\right)  =-1$. We now have%
\begin{gather*}
\eta\left(  x_0-1\right)=1, \quad 
\eta\left(  x_0\right)  =0, \quad
\eta\left(x_0+1\right)  =1, \quad
\eta\left(  x_0+2\right)  =1\\
\omega\left(  t-1,x_0-1\right)   =0.
\end{gather*}
Again TASEP dynamics does not change the values at $x_0$ and $x_0+1$. The
same argument applies to the case when $x_0+1$ is an arising pair point for
$\left(  (\theta,act)  ,\omega\right)  $.
\end{proof}

\section{Pure-Jump Weak Solutions of Burgers' Equation} \label{Sect Burgers}

We consider in this section Burgers' equation \eqref{eq:burgers},
\[
\partial_{t}u+u\partial_{x}u=0.
\]
We are interested in bounded (non entropic!) weak solutions, so we
restrict the definition to bounded functions, although it could be more general.

\begin{definition}\label{def weak solution}
	We say that a bounded measurable function $u:\left[t_0,t_1\right]  \times\mathbb{R\rightarrow R}$ 
	is a weak solution on $\left[  t_0,t_1\right]  $ if:
	\begin{itemize}
		\item[i)] for every smooth test function $\varphi:\mathbb{R\rightarrow R}$ with
		compact support in $\R$ the function $t\mapsto\int_\R u\left(  t,x\right)  \varphi(x)  dx$ 
		is continuous on $\left[ t_0,t_1\right]  $;
		\item[ii)] for every smooth test function $\phi:\R^2\to\R$
		with compact support in $\left(  t_0,t_1\right)  \times\R$, we
		have \
		\[
		\int_{t_0}^{t_1}\int_{\R}\left(  u\left(  t,x\right)  \partial
		_{t}\phi\left(  t,x\right)  +\frac12u^{2}\left(  t,x\right)  \partial
		_{x}\phi\left(  t,x\right)  \right)  dxdt=0.
		\]
	\end{itemize}
\end{definition}

Given a test function $\varphi$, the function $t\mapsto\int_{\R%
}u\left(  t,x\right)  \varphi(x)  dx$ is always defined almost everywhere, by
Fubini-Tonelli theorem. We require that it is continuous, for a minor reason
appearing in the next Proposition. It is not restrictive in our examples.

In the sequel we shall piece together weak solutions defined on different
space-time domains: let us see two rules allowing us to do so. When we say that $u\left(
\overline{t},\cdot\right)  =v\left(  \overline{t},\cdot\right)  $ for a
certain $\overline{t}\in\left[  t_0,t_1\right]  $ we mean that
$\int_{\R}u\left(  \overline{t},x\right)  \varphi(x)
dx=\int_{\R}v\left(  \overline{t},x\right)  \varphi(x)
dx$ for all test functions $\varphi$ of the class above.

\begin{proposition}
\label{prop merge in time}Assume $u\left(  t,x\right)  $ is a weak solution on
$\left[  t_0,t_1\right]  $ and $v\left(  t,x\right)  $ a weak solution on
$\left[  t_1,t_2\right]  $, with $u\left(  t_1,\cdot\right)  =v\left(
t_1,\cdot\right)  $. Then the function $w$, defined on $\left[  t_0%
,t_2\right]  $, equal to $u$ on $\left[  t_0,t_1\right]  $ and $v$ on
$\left[  t_1,t_2\right]  $, is a weak solution on $\left[  t_0%
,t_2\right]  $.
\end{proposition}

\begin{proof}
Let $\varphi:\R\to\R$ with compact support in $\R$.
Consider the function $t\mapsto\int_{\R}w\left(  t,x\right)
\varphi(x)  dx$, defined a.s. by Fubini-Tonelli theorem. By the
continuity of the function $t\mapsto\int_{\R}u\left(  t,x\right)
\varphi(x)  dx$ on $\left[  t_0,t_1\right]  $ and of
$t\mapsto\int_{\R}v\left(  t,x\right)  \varphi(x)  dx$ on
$\left[  t_1,t_2\right]  $ and by the property $u\left(  t_1%
,\cdot\right)  =v\left(  t_1,\cdot\right)  $, we deduce that $t\mapsto
\int_{\R}w\left(  t,x\right)  \varphi(x)  dx$ is continuous.

Given $\phi$ as in the definition, part (ii), introduce
\[
\phi_{n}\left(  t,x\right)  =\phi\left(  t,x\right)  \left(  1-\chi_{n}\left(
t-t_1\right)  \right)
\]
where $\chi_{n}\left(  s\right)  =\chi\left(  ns\right)$
and $\chi$ is smooth, $\chi(x)  =\chi\left(  -x\right)  $, with
values in $\left[  0,1\right]  $, equal to 1 in $\left[  -1,1\right]  $,
to zero outside $\left[  -2,2\right]  $; and take $n$ large enough. The
function $\phi_{n}\left(  t,x\right)  $, restricted to $\left(  t_0%
,t_1\right)  \times\R$, is a good test function for $u$, hence%
\[
\int_{t_0}^{t_1}\int_{\R}\left(  u\left(  t,x\right)  \partial
_{t}\phi_{n}\left(  t,x\right)  +\frac12u^{2}\left(  t,x\right)
\partial_{x}\phi_{n}\left(  t,x\right)  \right)  dxdt=0.
\]
Similarly on $\left(  t_1,t_2\right)  \times\R$ for $v$, hence%
\[
\int_{t_0}^{t_2}\int_{\R}\left(  w\left(  t,x\right)  \partial
_{t}\phi_{n}\left(  t,x\right)  +\frac12w^{2}\left(  t,x\right)
\partial_{x}\phi_{n}\left(  t,x\right)  \right)  dxdt=0.
\]
The same identity holds for $\phi$, completing the proof, if we show that
\begin{align*}
\lim_{n\rightarrow\infty}\int_{t_0}^{t_2}\int_{\R}w\left(
t,x\right)  \phi\left(  t,x\right)  \partial_{t}\chi_{n}\left(  t-t_1\right)  dxdt  & =0,\\
\lim_{n\rightarrow\infty}\int_{t_0}^{t_2}\int_{\R}w\left(
t,x\right)  \partial_{t}\phi\left(  t,x\right)  \chi_{n}\left(  t-t_1\right)  dxdt  & =0,\\
\lim_{n\rightarrow\infty}\int_{t_0}^{t_2}\int_{\R}w^{2}\left(
t,x\right)  \chi_{n}\left(  t-t_1\right)  \partial_{x}\phi\left(
t,x\right)  dxdt  & =0.
\end{align*}
The second and third limits are clear. The first claim is equivalent to%
\[
\lim_{n\rightarrow\infty}\int_{t_1-\frac{2}{n}}^{t_1+\frac{2}{n}}%
n\chi'\left(  n\left(  t-t_1\right)  \right)  \left(  \int%
_{\R}w\left(  t,x\right)  \phi\left(  t,x\right)  dx\right)  dt=0.
\]
It is easy, using also the boundedness of $w$, to show that the function
\begin{equation*}
	t\mapsto\int_{\R}w\left(  t,x\right)  \phi\left(  t,x\right)  dx
\end{equation*}
is continuous (approximate $\phi$ by functions piecewise constant in $t$).
With a similar argument we can replace this function by a constant in the
previous limit and thus reduce us to check the property
\[
\lim_{n\rightarrow\infty}\int_{-\frac{2}{n}}^{\frac{2}{n}}n\chi^{\prime
}\left(  nt\right)  dt=0
\]
(we have also changed variables). But this means
$\lim_{n\rightarrow\infty}\int_{-2}^{2}\chi'\left(  s\right)  ds=0$,
which is true by symmetry of $\chi$.
\end{proof}

\begin{proposition}
\label{prop merge in space-time}Assume that $u,v$ are two weak solutions, on
$\left[  t_0,t_1\right]  $, with disjoint supports, namely there are sets
$S_{u},S_{v}\subset\left[  t_0,t_1\right]  \times\R$, disjoint,
Borel measurable, such that $u=0$ a.s. outside $S_{u}$ and $v=0$ a.s. outside
$S_{v}$. Define
\[
w=u+v.
\]
Then $w$ is a weak solution. The result remains true when the intersection of
the supports has Lebesgue measure zero.
\end{proposition}

\begin{proof}
All properties are easily checked. Notice that $w^{2}=u^{2}+v^{2}$ almost everywhere.
\end{proof}
\noindent
Let us recall the definition of the Heaviside function and its formal weak
derivative
\begin{equation*}
H(x)   =1_{[0,\infty)}(x)  =1_{\left\{
x\geq0\right\}  },\quad 
H'(x)   =\delta_{0}(x)  =\delta\left(
x=0\right)  .
\end{equation*}
Given $t_0\in\R$, the simplest example of a pure-jump weak 
solution ---which is not an entropy solution---
of Burgers' equation on $\left[  t_0,t_1\right]  $, for any $t_1>$ $t_0$, is
\begin{equation*}
u\left(  t,x\right)  =1_{\left\{
x\geq x_0+v\left(  t-t_0\right)  \right\}  }w  
=1_{\left\{
x-x_0-v\left(  t-t_0\right)  \geq0\right\}  }w
=H\left(  x-x_0-v\left(t-t_0\right)  \right)  w
\end{equation*}
for $t\in\left[  t_0,t_1\right]$, and with $v>0$ and $w=2v$.
Part (i) of \autoref{def weak solution} comes from%
\[
\int_{\R}u\left(  t,x\right)  \varphi(x)  dx=\int%
_{x_0+v\left(  t-t_0\right)  }^{\infty}w\varphi(x)  dx,
\]
and this will be the case in all examples below, hence we shall not repeat it. 
Checking condition (ii) of \autoref{def weak solution} is elementary but quite lengthy.
However, we can perform a formal computation:
\begin{gather*}
\partial_{t}u  =-\delta_{0}\left(  x-x_0-v\left(  t-t_0\right)  \right) wv,\\
u^{2}\left(  t,x\right) =u\left(  t,x\right)  w,\ \quad
\partial_{x}u^{2}  =\delta_{0}\left(  x-x_0-v\left(  t-t_0\right)\right)  w^{2},
\end{gather*}
which, by $w^{2}=2vw$, implies $\partial_{x}u^{2}=-2\partial_{t}u$.
In what follows we will deal with analogous computations in more complicated situations:
the above formal computation with Dirac's deltas is both concise and transparent,
so we will keep on making use of it,
but any such computation can be easily made rigorous in terms of couplings with test functions.

\subsection{Isolated quasi-particles}

Given $h>0$ (it will be typically small in our main results), $v>0$,
$t_1>t_0$, we call \textit{right-quasi-particle} on $\left[  t_0%
,t_1\right]  $ a function of the following form: for $(t,x)\in\R\times [t_0,t_1]$, and $w=2v$,
\begin{align*}
u\left(  t,x\right) 
&=1_{\left\{  x_0+v\left(  t-t_0\right)  -h\leq x<x_0+v\left(
t-t_0\right)  \right\}  }w\\
& =1_{\left\{  x\geq x_0+v\left(  t-t_0\right)  -h\right\}  }w-1_{\left\{
x\geq x_0+v\left(  t-t_0\right)  \right\}  }w\\
& =H\left(  x-x_0-v\left(  t-t_0\right)  +h\right)  w-H\left(
x-x_0-v\left(  t-t_0\right)  \right)  w.
\end{align*}
The latter is a weak solution of Burgers' equation on $\left[  t_0,t_1\right]  $:
\begin{align*}
\partial_{t}u  & =-\delta_{0}\left(  x-x_0-v\left(  t-t_0\right)
+h\right)  wv+\delta_{0}\left(  x-x_0-v\left(  t-t_0\right)  \right)  wv,\\
u^{2}\left(  t,x\right)   & =u\left(  t,x\right)  w,\\
\partial_{x}u^{2}  & =\delta_{0}\left(  x-x_0-v\left(  t-t_0\right)
+h\right)  w^{2}-\delta_{0}\left(  x-x_0-v\left(  t-t_0\right)  \right)
w^{2},
\end{align*}
hence $\partial_{x}u^{2}=-2\partial_{t}u$.
Regarded as a soliton, a right-quasi-particle moves to the right with velocity $v$.

A \textit{left-quasi-particle} on $\left[  t_0,t_1\right]$ has the form
(with $v=w/2>0$ as above)
\begin{align*}
u\left(  t,x\right)
&=
-1_{\left\{  x_0-v\left(  t-t_0\right)  \leq x<x_0-v\left(
t-t_0\right)  +h\right\}  }w,\\
& =-\left(  1_{\left\{  x\geq x_0-v\left(  t-t_0\right)  \right\}
}-1_{\left\{  x\geq x_0-v\left(  t-t_0\right)  +h\right\}  }\right)  w,\\
& =-\left(  H\left(  x-x_0+v\left(  t-t_0\right)  \right)  -H\left(
x-x_0+v\left(  t-t_0\right)  -h\right)  \right)  w.
\end{align*}
Again, as a soliton, a left-quasi-particle moves to the left, again with velocity $v$.

\begin{figure}
	\centering
	\begin{tikzpicture}[line cap=round,line join=round,>=triangle 45,x=0.6cm,y=0.6cm]
	\clip(5,-3) rectangle (25,14);
	\fill[line width=0pt,dash pattern=on 3pt off 3pt,fill=black,fill opacity=0.25] (15,6) -- (15,2) -- (22,6.67) -- (22,10.67) -- cycle;
	\fill[line width=0pt,fill=black,fill opacity=0.25] (15,6) -- (16,8) -- (23,12.67) -- (22,10.67) -- cycle;
	\fill[line width=0pt,fill=black,fill opacity=0.1] (12.67,2.67) -- (15,-2) -- (15,2) -- (12.67,6.67) -- cycle;
	\fill[line width=0pt,fill=black,fill opacity=0.1] (12.67,6.67) -- (13.67,8.67) -- (15,6) -- (14,4) -- cycle;
	\fill[line width=0pt,fill=black,fill opacity=0.1] (14,4) -- (15,6) -- (15,2) -- cycle;
	\fill[line width=0pt,fill=black,fill opacity=0.1] (15,2) -- (16,2.67) -- (16,0) -- (15,-2) -- cycle;
	\draw [->] (7,0.5) -- (24,0.5);
	\draw [->] (7,0.5) -- (11,8.5);
	\draw [->] (7,0.5) -- (7,11.5);
	\draw (24,1) node {$x$};
	\draw (11,9) node {$t$};
	\draw (7,12) node {$u$};
	\draw [dotted] (14.25,0.5)-- (20,12);
	\draw [dotted] (7.75,2)-- (22,2);
	\draw (15,2)-- (15,6);
	\draw (15,6)-- (16,8);
	\draw (15,6)-- (22,10.67);
	\draw (15,2)-- (22,6.67);
	\draw [dash pattern=on 3pt off 3pt] (16,8)-- (16,4);
	\draw [dash pattern=on 3pt off 3pt] (16,4)-- (16,2.67);
	\draw (16,2.67)-- (16,0);
	\draw (15,2)-- (15,-2);
	\draw (15,-2)-- (16,0);
	\draw (16,8)-- (23,12.67);
	\draw (12.67,2.67)-- (15,-2);
	\draw (13.67,8.67)-- (15,6);
	\draw [dash pattern=on 3pt off 3pt] (15,6)-- (16,4);
	\draw (15,-2)-- (16,0);
	\draw (16,0)-- (16,2.67);
	\draw (16,2.67)-- (15,2);
	\draw (15,2)-- (15,-2);
	\draw (15,2)-- (12.67,6.67);
	\draw [dash pattern=on 3pt off 3pt] (16,4)-- (22,8);
	\end{tikzpicture}
	\caption{Generation of a couple of right and left quasi-particles,
		in two different shades of gray.}
	\label{fig:quasiparticles} 
\end{figure}
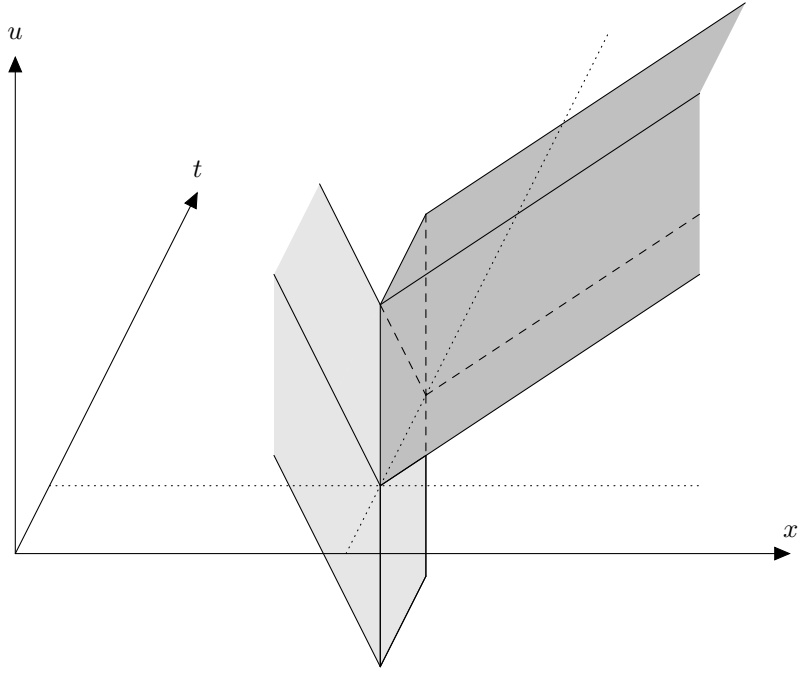

\subsection{Arising pairs of quasi-particles}
In our construction, the traveling solitons just defined usually emerge
somewhere and disappear somewhere else. 
In this subsection we describe the creation mechanism: 
quasi-particles appear in pairs, a positive and a negative one, moving in
opposite directions. After a short time of order $h$ they become isolated
solitons of the form described in the previous subsection. But at the
beginning, when they emerge and develop, they are made of two pieces with
increasing size smaller than $h$, see \autoref{fig:quasiparticles}.

\begin{remark}
	An arising pair comes from the
	identically zero solution. Hence, before the birth time $t_0$, we have $u=0$,
	which is a weak solution. In the time interval $\left[  t_0,t_0+\frac
	{h}{v}\right]  $ the pair develops. After time $t_0+\frac{h}{v}$ we have two
	disjoint isolated quasi-particles, which is again a weak solution, by
	\autoref{prop merge in space-time} and the examples of the previous
	subsection. Thus, thanks to \autoref{prop merge in time}, it is
	sufficient to define the arising pair in $\left[  t_0,t_0+\frac{h}%
	{v}\right]  $ and to prove that it is a weak solution there.
\end{remark}

Given $h>0$, $v>0$ and $t_0$, we call \textit{arising pair of quasi-particles
at }$\left(  t_0,x_0\right)  $, the function defined
for  $(t,x)\in \left[ t_0,t_0+\frac{h}{v}\right]\times \R $ by
\[
u\left(  t,x\right)  =1_{\left\{  x_0\leq x<x_0+v\left(  t-t_0\right)  \right\}
}w-1_{\left\{  x_0-v\left(  t-t_0\right)  \leq x<x_0\right\}  }w
\]
with $w=2v$. 
An expression for continuing the motion after $t_0+\frac{h}{v}$ is
\begin{multline*}
u\left(  t,x\right)  
=1_{\left\{  x_0+v\left(  t-t_0\right)
-\min\left(  h,v\left(  t-t_0\right)  \right)  \leq x<x_0+v\left(
t-t_0\right)  \right\}  }w\\
 -1_{\left\{  x_0-v\left(  t-t_0\right)  \leq x<x_0-v\left(
t-t_0\right)  +\min\left(  h,v\left(  t-t_0\right)  \right)  \right\}  }w.
\end{multline*}
Let us treat only the case $t\in\left[  t_0,t_0+\frac{h}{v}\right]  $. It
can also be written as%
\begin{align*}
u\left(  t,x\right)  
&=1_{\left\{  x\geq x_0\right\}  }w-1_{\left\{  x\geq
x_0+v\left(  t-t_0\right)  \right\}  }w-\left(  1_{\left\{  x\geq
x_0-v\left(  t-t_0\right)  \right\}  }-1_{\left\{  x\geq x_0\right\}
}\right)  w\\
& =H\left(  x-x_0\right)  w-H\left(  x-x_0-v\left(  t-t_0\right)
\right)  w\\
&\quad -\left(  H\left(  x-x_0+v\left(  t-t_0\right)  \right)  -H\left(
x-x_0\right)  \right)  w.
\end{align*}
It is a weak solution of Burgers' equation: indeed it holds
\begin{equation*}
	\partial_{t}u=\delta_{0}\left(  x-x_0-v\left(  t-t_0\right)  \right)
	wv-\delta_{0}\left(  x-x_0+v\left(  t-t_0\right)  \right)  wv
\end{equation*}
and, using the fact that the two pieces have disjoint support,
\begin{align*}
	u^{2}\left(  t,x\right)  
	&=1_{\left\{  x_0\leq x\leq x_0+v\left(
		t-t_0\right)  \right\}  }w^{2}+1_{\left\{  x_0-v\left(  t-t_0\right)
		\leq x\leq x_0\right\}  }w^{2}\\
& =H\left(  x-x_0\right)  w^{2}-H\left(  x-x_0-v\left(  t-t_0\right)
\right)  w^{2}\\
& +H\left(  x-x_0+v\left(  t-t_0\right)  \right)  w^{2}-H\left(
x-x_0\right)  w^{2},\\
\partial_{x}u^{2}  & =\delta_{0}\left(  x-x_0\right)  w^{2}-\delta
_{0}\left(  x-x_0-v\left(  t-t_0\right)  \right)  w^{2}\\
& +\delta_{0}\left(  x-x_0+v\left(  t-t_0\right)  \right)  w^{2}%
-\delta_{0}\left(  x-x_0\right)  w^{2} =-2\partial_{t}u.
\end{align*}

\subsection{Coalescing pairs of quasi-particles}
As anticipated above, usually a quasi-particle meets after a short time another
quasi-particle traveling in the opposite direction: in this case they annihilate each other. 
This process is described by the following solution:
given $h>0$, $v=w/2>0$ and $t_1$, we call \textit{coalescing pair of
quasi-particles at }$\left(  t_1,x_0\right)  $, the function defined for
$(t,x)\in \left[  t_1-\frac{h}{v},t_1\right]\times\R$ by
\begin{equation}\label{coalescing pairs}
u\left(  t,x\right) =1_{\left\{  x_0-v\left(  t_1-t\right)  \leq x<x_0\right\}
}w-1_{\left\{  x_0\leq x<x_0+v\left(  t_1-t\right)  \right\}}w.
\end{equation}
The proof that it is a weak solution is the same as for arising
quasi-particles: in fact one can also argue by time-reversal of Burgers' equation. 

\begin{remark}\label{rmk:paramlambda}
	The content of the present section is easily adapted to produce weak solutions of
	\begin{equation*}
		\partial_{t}u=\lambda \partial_{x}u^2
	\end{equation*}
	with $\lambda\neq 0$, the latter being the formal derivative of \eqref{eq:hamiltonjacobi}.
	Indeed, it suffices to replace the relation $w=2v$ between parameters $v,w$
	with $w=-\lambda v$.
\end{remark}

\section{TASEP pairs, ABDF and Burgers quasi-particles}\label{Sect Burgers TASEP ABDF}

In this section we describe a bijection between TASEP sample (and
therefore ABDF samples) and realizations of a random weak solution of
Burgers' equation. We associate special configurations of
Burgers solutions to ABDF configurations, using integer times $t_0\in\N$ for this correspondence;
then we interpolate for $t\in[t_0,t_0+1]$ using the special quasi-particle weak solutions
defined in the previous section.

This idea however is complicated by a tricky detail. In the ABDF model,
creation of new pairs happens at integer times $t_0-1$ (without being
visible) and is observed only at time $t_0$: discrete time allows to do so.
We refer to \autoref{fig:abdfdynamics} for an example:
new pairs arise from empty sites (diamond-shaped in the picture).
On the contrary, due to continuous time, for Burgers' equation we need to
create new pairs before integer times, so that the pair is fully formed at
integer time. The creation instant of a new pair will be at times
$t_0-1-\frac12$, $t_0\in\N_0$.

Strictly speaking, the correspondence between TASEP and Burgers' samples is not a conjugation of random
dynamical systems, as opposed to the one between
TASEP and ABDF, because in order to define the configuration at time
$t_0\in\N$ of the random weak solution of Burgers' equation we need
two pieces of information: the configuration of TASEP -- or ABDF -- 
at time $t_0$ and the noise values $\set{\omega (t_0,x)  ;x\in\Z}$. 
Nevertheless, it is a bijection of samples of stochastic processes, and thus it may allow to study the behavior of
each one of the two processes starting from the other one.

\begin{definition}\label{Def u_0}
Given $t_0\in\N$, $(\theta,act)\in\Lambda$, $\omega\in\Omega$, 
let us introduce the following sets:
\begin{align*}
M^{+}\left(  \theta\right)   
& =\left\{  z\in\Z:\theta\left(
z\right)  =1\right\} \\
M^{-}\left(  \theta\right)   
& =\left\{  z'\in\Z:\theta\left(
z'\right)  =-1\right\},\\
A\left(  \theta,act,\omega,t_0\right)  
&=\left\{  z\in\Z:act\left(
z\right)  =1,\omega\left(  t_0,z\right)  =0\right\},\\
MA^{+}\left(  \theta,act,\omega,t_0\right)   
& =M^{+}\left(  \theta\right)
\cup A\left(  \theta,act,\omega,t_0\right), \\
MA^{-}\left(  \theta,act,\omega,t_0\right)   
& =M^{-}\left(  \theta\right)
\cup A\left(  \theta,act,\omega,t_0\right),\\
C\left(  \theta,act,\omega,t_0\right)&=
\left\{ z\in\Z:z-1\in MA^{+}\left(  \theta,act,\omega
,t_0\right),\right.\\
&\qquad \left.  z+1\in MA^{-}\left(  \theta,act,\omega,t_0\right)  \right\}.
\end{align*}
(In plain words they are the sets of, respectively,  particles Moving to the
right, particles Moving to the left, sites of Arising pairs, sites of Moving
or Arising, sites of Coalescence).

Given $(\theta,act)  \in\Lambda$, $t_0\in\N$ and $\omega\in\Omega$, we define, for $x\in\R$,%
\begin{align*}
u(t_0,x,\omega)   &=2\sum_{z\in M^{+}\left(  \theta\right)
}1_{\left\{  z\leq x<z+\frac12\right\}  }-2\sum_{z\in M^{-}\left(
\theta\right)  }1_{\left\{  z-\frac12\leq x<z\right\}  }\\
& \quad +2\sum_{z\in A\left(  \theta,act,\omega,t_0\right)  }\left(  1_{\left\{
z\leq x<z+\frac12\right\}  }-1_{\left\{  z-\frac12\leq x<z\right\}
}\right)  .
\end{align*}%
\end{definition}

\begin{figure}
	\centering
	\begin{tikzpicture}[line cap=round,line join=round,>=triangle 45,x=0.8cm,y=0.8cm]
	\clip(-1,-2) rectangle (20,13);
	\fill[fill=black,fill opacity=0.25] (1,0) -- (1,2) -- (1.5,2) -- (1.5,0) -- cycle;
	\fill[fill=black,fill opacity=0.25] (3,0) -- (3,2) -- (3.5,2) -- (3.5,0) -- cycle;
	\fill[fill=black,fill opacity=0.1] (3,0) -- (2.5,0) -- (2.5,-2) -- (3,-2) -- cycle;
	\fill[fill=black,fill opacity=0.25] (6,0) -- (6,2) -- (6.5,2) -- (6.5,0) -- cycle;
	\fill[fill=black,fill opacity=0.1] (8,0) -- (7.5,0) -- (7.5,-2) -- (8,-2) -- cycle;
	\fill[fill=black,fill opacity=0.1] (9,0) -- (8.5,0) -- (8.5,-2) -- (9,-2) -- cycle;
	\fill[fill=black,fill opacity=0.1] (12,0) -- (11.5,0) -- (11.5,-2) -- (12,-2) -- cycle;
	\fill[fill=black,fill opacity=0.25] (12,0) -- (12,2) -- (12.5,2) -- (12.5,0) -- cycle;
	\fill[fill=black,fill opacity=0.25] (13,0) -- (13,2) -- (13.5,2) -- (13.5,0) -- cycle;
	\fill[fill=black,fill opacity=0.25] (14,0) -- (14,2) -- (14.5,2) -- (14.5,0) -- cycle;
	\fill[fill=black,fill opacity=0.25] (4,5) -- (4,7) -- (4.5,7) -- (4.5,5) -- cycle;
	\fill[fill=black,fill opacity=0.1] (8,5) -- (8,3) -- (7.5,3) -- (7.5,5) -- cycle;
	\fill[fill=black,fill opacity=0.1] (9,5) -- (8.5,5) -- (8.5,3) -- (9,3) -- cycle;
	\fill[fill=black,fill opacity=0.25] (9,5) -- (9,7) -- (9.5,7) -- (9.5,5) -- cycle;
	\fill[fill=black,fill opacity=0.1] (11,5) -- (11,3) -- (10.5,3) -- (10.5,5) -- cycle;
	\fill[fill=black,fill opacity=0.25] (13,5) -- (13,7) -- (13.5,7) -- (13.5,5) -- cycle;
	\fill[fill=black,fill opacity=0.25] (14,5) -- (14,7) -- (14.5,7) -- (14.5,5) -- cycle;
	\fill[fill=black,fill opacity=0.25] (5,10) -- (5,12) -- (5.5,12) -- (5.5,10) -- cycle;
	\fill[fill=black,fill opacity=0.1] (7,10) -- (6.5,10) -- (6.5,8) -- (7,8) -- cycle;
	\fill[fill=black,fill opacity=0.1] (8,10) -- (7.5,10) -- (7.5,8) -- (8,8) -- cycle;
	\fill[fill=black,fill opacity=0.1] (11,10) -- (10.5,10) -- (10.5,8) -- (11,8) -- cycle;
	\fill[fill=black,fill opacity=0.25] (13,10) -- (13,12) -- (13.5,12) -- (13.5,10) -- cycle;
	\fill[fill=black,fill opacity=0.25] (14,10) -- (14,12) -- (14.5,12) -- (14.5,10) -- cycle;
	\fill[fill=black,fill opacity=0.25] (12,5) -- (12,7) -- (12.5,7) -- (12.5,5) -- cycle;
	\fill[fill=black,fill opacity=0.1] (12,5) -- (11.5,5) -- (11.5,3) -- (12,3) -- cycle;
	\draw [->] (0,0) -- (15,0);
	\draw [->] (0,5) -- (15,5);
	\draw [->] (0,10) -- (15,10);
	\draw (0,0.3) node {$t=0$};
	\draw (0,5.3) node {$t=1$};
	\draw (0,10.3) node {$t=2$};
	\draw (1,0)-- (1,2);
	\draw (1,2)-- (1.5,2);
	\draw (1.5,2)-- (1.5,0);
	\draw (1.5,0)-- (1,0);
	\draw (3,0)-- (3,2);
	\draw (3,2)-- (3.5,2);
	\draw (3.5,2)-- (3.5,0);
	\draw (3.5,0)-- (3,0);
	\draw (3,0)-- (2.5,0);
	\draw (2.5,0)-- (2.5,-2);
	\draw (2.5,-2)-- (3,-2);
	\draw (3,-2)-- (3,0);
	\draw (6,0)-- (6,2);
	\draw (6,2)-- (6.5,2);
	\draw (6.5,2)-- (6.5,0);
	\draw (6.5,0)-- (6,0);
	\draw (8,0)-- (7.5,0);
	\draw (7.5,0)-- (7.5,-2);
	\draw (7.5,-2)-- (8,-2);
	\draw (8,-2)-- (8,0);
	\draw (9,0)-- (8.5,0);
	\draw (8.5,0)-- (8.5,-2);
	\draw (8.5,-2)-- (9,-2);
	\draw (9,-2)-- (9,0);
	\draw (12,0)-- (11.5,0);
	\draw (11.5,0)-- (11.5,-2);
	\draw (11.5,-2)-- (12,-2);
	\draw (12,-2)-- (12,0);
	\draw (12,0)-- (12,2);
	\draw (12,2)-- (12.5,2);
	\draw (12.5,2)-- (12.5,0);
	\draw (12.5,0)-- (12,0);
	\draw (13,0)-- (13,2);
	\draw (13,2)-- (13.5,2);
	\draw (13.5,2)-- (13.5,0);
	\draw (13.5,0)-- (13,0);
	\draw (14,0)-- (14,2);
	\draw (14,2)-- (14.5,2);
	\draw (14.5,2)-- (14.5,0);
	\draw (14.5,0)-- (14,0);
	\draw (4,5)-- (4,7);
	\draw (4,7)-- (4.5,7);
	\draw (4.5,7)-- (4.5,5);
	\draw (4.5,5)-- (4,5);
	\draw (8,5)-- (8,3);
	\draw (8,3)-- (7.5,3);
	\draw (7.5,3)-- (7.5,5);
	\draw (7.5,5)-- (8,5);
	\draw (9,5)-- (8.5,5);
	\draw (8.5,5)-- (8.5,3);
	\draw (8.5,3)-- (9,3);
	\draw (9,3)-- (9,5);
	\draw (9,5)-- (9,7);
	\draw (9,7)-- (9.5,7);
	\draw (9.5,7)-- (9.5,5);
	\draw (9.5,5)-- (9,5);
	\draw (11,5)-- (11,3);
	\draw (11,3)-- (10.5,3);
	\draw (10.5,3)-- (10.5,5);
	\draw (10.5,5)-- (11,5);
	\draw (13,5)-- (13,7);
	\draw (13,7)-- (13.5,7);
	\draw (13.5,7)-- (13.5,5);
	\draw (13.5,5)-- (13,5);
	\draw (14,5)-- (14,7);
	\draw (14,7)-- (14.5,7);
	\draw (14.5,7)-- (14.5,5);
	\draw (14.5,5)-- (14,5);
	\draw (5,10)-- (5,12);
	\draw (5,12)-- (5.5,12);
	\draw (5.5,12)-- (5.5,10);
	\draw (5.5,10)-- (5,10);
	\draw (7,10)-- (6.5,10);
	\draw (6.5,10)-- (6.5,8);
	\draw (6.5,8)-- (7,8);
	\draw (7,8)-- (7,10);
	\draw (8,10)-- (7.5,10);
	\draw (7.5,10)-- (7.5,8);
	\draw (7.5,8)-- (8,8);
	\draw (8,8)-- (8,10);
	\draw (11,10)-- (10.5,10);
	\draw (10.5,10)-- (10.5,8);
	\draw (10.5,8)-- (11,8);
	\draw (11,8)-- (11,10);
	\draw (13,10)-- (13,12);
	\draw (13,12)-- (13.5,12);
	\draw (13.5,12)-- (13.5,10);
	\draw (13.5,10)-- (13,10);
	\draw (14,10)-- (14,12);
	\draw (14,12)-- (14.5,12);
	\draw (14.5,12)-- (14.5,10);
	\draw (14.5,10)-- (14,10);
	\draw (12,5)-- (12,7);
	\draw (12,7)-- (12.5,7);
	\draw (12.5,7)-- (12.5,5);
	\draw (12.5,5)-- (12,5);
	\draw (12,5)-- (11.5,5);
	\draw (11.5,5)-- (11.5,3);
	\draw (11.5,3)-- (12,3);
	\draw (12,3)-- (12,5);
	\begin{scriptsize}
	\fill [color=black] (1,0) circle (2.5pt);
	\fill [color=black] (6,0) circle (2.5pt);
	\draw [color=black] (3,0) ++(-2.5pt,0 pt) -- ++(2.5pt,2.5pt)--++(2.5pt,-2.5pt)--++(-2.5pt,-2.5pt)--++(-2.5pt,2.5pt);
	\draw [color=black] (2,0) circle (2.5pt);
	\draw [color=black] (4,0) circle (2.5pt);
	\draw [color=black] (5,0) circle (2.5pt);
	\draw [color=black] (7,0) circle (2.5pt);
	\fill [color=black] (8,0) circle (2.5pt);
	\fill [color=black] (9,0) circle (2.5pt);
	\draw [color=black] (10,0) circle (2.5pt);
	\draw [color=black] (11,0) circle (2.5pt);
	\draw [color=black] (12,0) ++(-2.5pt,0 pt) -- ++(2.5pt,2.5pt)--++(2.5pt,-2.5pt)--++(-2.5pt,-2.5pt)--++(-2.5pt,2.5pt);
	\fill [color=black] (13,0) circle (2.5pt);
	\fill [color=black] (14,0) circle (2.5pt);
	\draw [color=black] (1,5) circle (2.5pt);
	\draw [color=black] (2,5) circle (2.5pt);
	\draw [color=black] (3,5) circle (2.5pt);
	\fill [color=black] (4,5) circle (2.5pt);
	\draw [color=black] (5,5) circle (2.5pt);
	\draw [color=black] (6,5) circle (2.5pt);
	\draw [color=black] (7,5) circle (2.5pt);
	\fill [color=black] (8,5) circle (2.5pt);
	\draw [color=black] (9,5) ++(-2.5pt,0 pt) -- ++(2.5pt,2.5pt)--++(2.5pt,-2.5pt)--++(-2.5pt,-2.5pt)--++(-2.5pt,2.5pt);
	\draw [color=black] (10,5) circle (2.5pt);
	\fill [color=black] (11,5) circle (2.5pt);
	\draw [color=black] (12,5) ++(-2.5pt,0 pt) -- ++(2.5pt,2.5pt)--++(2.5pt,-2.5pt)--++(-2.5pt,-2.5pt)--++(-2.5pt,2.5pt);
	\fill [color=black] (13,5) circle (2.5pt);
	\fill [color=black] (14,5) circle (2.5pt);
	\draw [color=black] (1,10) circle (2.5pt);
	\draw [color=black] (2,10) circle (2.5pt);
	\draw [color=black] (3,10) circle (2.5pt);
	\draw [color=black] (4,10) circle (2.5pt);
	\fill [color=black] (5,10) circle (2.5pt);
	\draw [color=black] (6,10) circle (2.5pt);
	\fill [color=black] (7,10) circle (2.5pt);
	\fill [color=black] (8,10) circle (2.5pt);
	\draw [color=black] (9,10) circle (2.5pt);
	\draw [color=black] (10,10) circle (2.5pt);
	\fill [color=black] (11,10) circle (2.5pt);
	\draw [color=black] (12,10) circle (2.5pt);
	\fill [color=black] (13,10) circle (2.5pt);
	\fill [color=black] (14,10) circle (2.5pt);
	\end{scriptsize}
	\end{tikzpicture}
	\caption{Profile of $u$ defined by \autoref{Def u_0}, starting from ABDF configuration
	at times $t=0,1,2$ of \autoref{fig:abdfdynamics}.}
	\label{fig:uintegertimes} 
\end{figure}
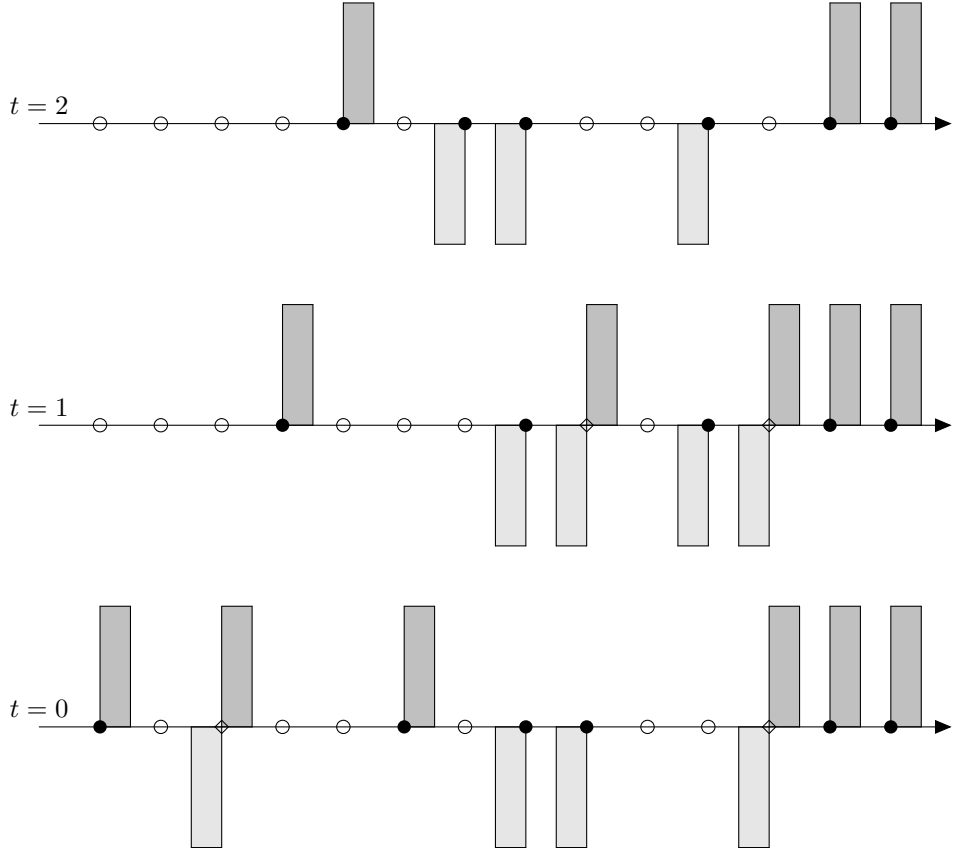
\noindent
To explain the definition, assume $(\theta,act)=\PP(\eta)$. 
The formula for $u (t_0,x)  $ includes three summands:
\begin{itemize}
\item a right-quasi-particle%
\[
2\cdot1_{\left\{  z\leq x<z+\frac12\right\}  }%
\]
at each point $z$ where $\theta\left(  z\right)  =1$, namely where there is a
right-particle of ABDF, or equivalently where TASEP has two consecutive empty spaces;

\item a left-quasi-particle%
\[
-2\cdot1_{\left\{  z-\frac12\leq x<z\right\}  }%
\]
at each point $z$ where $\theta\left(  z\right)  =-1$, namely where there is a
left-particle of ABDF, or equivalently where TASEP has two consecutive particles;

\item an arising pair of quasi-particles%
\[
2\left(  1_{\left\{  z\leq x<z+\frac12\right\}  }-1_{\left\{  z-\frac
{1}{2}\leq x<z\right\}  }\right)
\]
at each point $z$ where $\theta\left(  z\right)  =0$, $act\left(  z\right)
=1$ (or equivalently $\eta\left(  z\right)  =1$, $\eta\left(  z+1\right)  =0$)
and $\omega (t_0,x)  =0$, namely at each active empty position
of ABDF, where the noise prescribes the creation of two particles.
\end{itemize}
\noindent
We have chosen $h=\frac12$, $v=1$, and thus $w=2$, in the definitions of
quasi-particles of \autoref{Sect Burgers}; the value of $h$ can be
changed to any value in $\left(0,1\right)$ without consequences, while the value of
$v$ is coordinated with the scheme.

\subsection{Bijection with TASEP at Integer Times}
In the remainder of this section we need to show several facts. The first one is
the bijection property between TASEP (or ABDF) realizations and these
particular functions $u\left(t_0,x,\omega\right)$. This can be formalized
in different ways; we limit ourselves to state that, given the function
$u(t_0,x,\omega)  $, above, we can reconstruct the values of
$(\theta,act)$. The proof is straightforward, just noticing that each $z\in
\Z$ appears at most in one of the sums defining $u\left(
t_0,x,\omega\right)  $.

\begin{proposition}\label{prop theta from u}
Let $u(t_0,x,\omega)  $ be given by \autoref{Def u_0}, 
with respect to $(\theta,act)\in\Lambda$ and $\omega\in\Omega$. Then
\[
\theta\left(  z\right)  =\int_{z-\frac12}^{z+\frac12}u\left(
t_0,x,\omega\right)  dx.
\]
If $\theta\neq0$, $act$ is $ar\left(  \theta\right)  $ and thus can be
reconstructed from $u$. If $\theta=0$, for a.e. $\omega$ there are infinitely
many active points $z$ where $\omega\left(  t_0,z\right)  =0$; this means
that $u(t_0,x,\omega)  $ contains infinitely many points $z$
where the jump
\[
\left[  u\left(  t_0,\cdot,\omega\right)  \right]  _{z}:=\lim_{x\rightarrow
z^{+}}u(t_0,x,\omega)  -\lim_{x\rightarrow z^{-}}u\left(
t_0,x,\omega\right)
\]
is equal to 4. If these points are even, $act=alt_1$, otherwise
$act=alt_0$.
\end{proposition}

\subsection{Continuation Shortly After Integer Times}
Next, we have to interpolate the functions $u(t_0,x,\omega)$ between integer times. The particular weak solutions introduced
in \autoref{Sect Burgers} prescribe a unique continuation of $u(t_0,x,\omega)$ from the ``initial''
 value at time $t_0\in\N$ to all values of $t$ in $\left[
t_0,t_0+\frac12\right]  $ (the value $\frac12$ is related to the
choice $h=\frac12$):
\begin{align*}
u\left(  t,x,\omega\right)   & =2\sum_{z\in MA^{+}\left(  \theta
,act,\omega,t_0\right)  }1_{\left\{  z+\left(  t-t_0\right)  \leq
x<z+\left(  t-t_0\right)  +\frac12\right\}  }\\
& -2\sum_{z'\in MA^{-}\left(  \theta,act,\omega,t_0\right)
}1_{\left\{  z'-\left(  t-t_0\right)  -\frac12\leq x<z^{\prime
}-\left(  t-t_0\right)  \right\}  }.
\end{align*}%

\begin{proposition}\label{prop first 1/2}
The function thus defined for $t\in\left[  t_0,t_0+\frac12\right]$, $x\in\R$ 
is a weak solution of Burgers' equation.
\end{proposition}

\begin{proof}
Coincidence of the last formula at time $t_0$ with the initial condition
$u(t_0,x,\omega)  $ above is obvious. The statement is a consequence of a simple fact:
every pair of terms taken from the two sums defining $u\left(  t,x,\omega\right)  $ is made of quasi-particles
with disjoint supports on $\left[  t_0,t_0+\frac12\right]  $, and thus
the sum solves Burgers' equation in weak sense --by \autoref{prop merge in space-time} --
on the interval $\left[  t_0,t_0+\frac12\right]  $.

Let us check that supports are disjoint.
Quasi-particles moving to the right (those corresponding to the first sum) are
clearly isolated between themselves, having a ``support'' of size $\frac12$ of the form $[x\left(
t\right)  ,x\left(  t\right)  +\frac12)$ with $x\left(  t\right)  $ of the
form $z+\left(  t-t_0\right)  $ with $z$ of distance at least one from each other.
The same holds for left-quasi-particles, among themselves. 
Thus the problem is only about the interaction between a right-quasi-particle
\[
2\cdot1_{\left\{  z+\left(  t-t_0\right)  \leq x<z+\left(  t-t_0\right)
+\frac12\right\}  }%
\]
and a left-quasi-particle%
\[
-2\cdot1_{\left\{  z'-\left(  t-t_0\right)  -\frac12\leq
x<z'-\left(  t-t_0\right)  \right\}  }%
\]
with $z$ in the first sum and $z'$ in the second one. The supports of
these two solitons have size $\frac12$ and are of the form $[x\left(
t\right)  ,x\left(  t\right)  +\frac12)$ with $x\left(  t\right)
=z+\left(  t-t_0\right)  $ and $[x'\left(  t\right)  -\frac{1}%
{2},x'\left(  t\right)  )$ with $x'\left(  t\right)
=z'-\left(  t-t_0\right)  $, respectively. We claim that these
supports are disjoint, for $t\in\left[  t_0,t_0+\frac12\right]  $. If
$z'\leq z$ this is clear, since $x'\left(  t\right)  $ is
decreasing and $x\left(  t\right)  $ is increasing. 
When $z'>z$ we claim that sets $[x\left(  t\right)  ,x\left(  t\right)  +\frac12)$
and $[x'\left(  t\right)  -\frac12,x'\left(  t\right)  )$
are disjoint because $\left(  t-t_0\right)  +\frac12\leq1$
and $z'\geq z+2$
(to be shown below) and thus
\[
z'-\left(  t-t_0\right)  -\frac12\geq z+\left(  t-t_0\right)
+\frac12.
\]
The key fact $z'\geq z+2$ requires inspection into the conditions that
$z$ and $z'$ belong to two different sums. Recall we are treating the
case $z'>z$, hence the two solitons are not the result of an arising pair.

We have $z\in MA^{+}\left(  \theta,act,\omega,t_0\right)  $. This is the
union of two cases. Consider the case $\theta\left(  z\right)  =1$ and assume
by contradiction that $z'=z+1$. By the rules of $\Lambda$, $\theta\left(  z'\right)  $ cannot be $-1$ 
(because the number of integer points strictly between $z$ and $z'$ is even);
by the rules of the map $ar$, if $\theta\left(  z'\right)  =0$, $ar\left(  \theta\right)  \left(
z'\right)  $ is zero. Hence we have found a contradiction.

Consider the case $act\left(  z\right)  =1$ and assume by contradiction that
$z'=z+1$. Again by the rules of $\Lambda$ we cannot have
$\theta\left(  z'\right)  =-1$ and we cannot have $act\left(
z'\right)  =1$. Hence, we get a contradiction also in this case, and this
proves $z'\geq z+2$.
\end{proof}

\begin{figure}
	\centering
	\begin{tikzpicture}[line cap=round,line join=round,>=triangle 45,x=0.8cm,y=0.8cm]
	\clip(-1,-1) rectangle (16,4);
	\fill[fill=black,fill opacity=0.25] (0,0) -- (1,1) -- (1,0.5) -- (0.5,0) -- cycle;
	\fill[fill=black,fill opacity=0.1] (1,0.5) -- (1.5,0) -- (2,0) -- (1,1) -- cycle;
	\fill[fill=black,fill opacity=0.25] (2,0) -- (4.5,2.5) -- (5,2.5) -- (2.5,0) -- cycle;
	\fill[fill=black,fill opacity=0.25] (5,0) -- (6,1) -- (6,0.5) -- (5.5,0) -- cycle;
	\fill[fill=black,fill opacity=0.1] (6,1) -- (7,0) -- (6.5,0) -- (6,0.5) -- cycle;
	\fill[fill=black,fill opacity=0.1] (7.5,0) -- (5,2.5) -- (5.5,2.5) -- (8,0) -- cycle;
	\fill[fill=black,fill opacity=0.1] (8,0.5) -- (6,2.5) -- (6.5,2.5) -- (8,1) -- cycle;
	\fill[fill=black,fill opacity=0.25] (8,1) -- (9,2) -- (9,1.5) -- (8,0.5) -- cycle;
	\fill[fill=black,fill opacity=0.1] (9,2) -- (11,0) -- (10.5,0) -- (9,1.5) -- cycle;
	\fill[fill=black,fill opacity=0.1] (11,1) -- (9.5,2.5) -- (10,2.5) -- (11,1.5) -- cycle;
	\fill[fill=black,fill opacity=0.25] (11,1.5) -- (12,2.5) -- (12.5,2.5) -- (11,1) -- cycle;
	\fill[fill=black,fill opacity=0.25] (11,0) -- (13.5,2.5) -- (14,2.5) -- (11.5,0) -- cycle;
	\fill[fill=black,fill opacity=0.25] (12,0) -- (14,2) -- (14,1.5) -- (12.5,0) -- cycle;
	\fill[fill=black,fill opacity=0.25] (13,0) -- (14,1) -- (14,0.5) -- (13.5,0) -- cycle;
	\draw (0,0)-- (1,1);
	\draw (1,1)-- (1,0.5);
	\draw (1,0.5)-- (0.5,0);
	\draw (0.5,0)-- (0,0);
	\draw (1,0.5)-- (1.5,0);
	\draw (1.5,0)-- (2,0);
	\draw (2,0)-- (1,1);
	\draw (1,1)-- (1,0.5);
	\draw (2,0)-- (4.5,2.5);
	\draw (5,2.5)-- (2.5,0);
	\draw (2.5,0)-- (2,0);
	\draw (5,0)-- (6,1);
	\draw (6,1)-- (6,0.5);
	\draw (6,0.5)-- (5.5,0);
	\draw (5.5,0)-- (5,0);
	\draw (6,1)-- (7,0);
	\draw (7,0)-- (6.5,0);
	\draw (6.5,0)-- (6,0.5);
	\draw (6,0.5)-- (6,1);
	\draw (7.5,0)-- (5,2.5);
	\draw (5.5,2.5)-- (8,0);
	\draw (8,0)-- (7.5,0);
	\draw (8,0.5)-- (6,2.5);
	\draw (6.5,2.5)-- (8,1);
	\draw (8,1)-- (8,0.5);
	\draw (8,1)-- (9,2);
	\draw (9,2)-- (9,1.5);
	\draw (9,1.5)-- (8,0.5);
	\draw (8,0.5)-- (8,1);
	\draw (9,2)-- (11,0);
	\draw (11,0)-- (10.5,0);
	\draw (10.5,0)-- (9,1.5);
	\draw (9,1.5)-- (9,2);
	\draw (11,1)-- (9.5,2.5);
	\draw (10,2.5)-- (11,1.5);
	\draw (11,1.5)-- (11,1);
	\draw (11,1.5)-- (12,2.5);
	\draw (12.5,2.5)-- (11,1);
	\draw (11,1)-- (11,1.5);
	\draw (11,0)-- (13.5,2.5);
	\draw (14,2.5)-- (11.5,0);
	\draw (11.5,0)-- (11,0);
	\draw (12,0)-- (14,2);
	\draw (14,1.5)-- (12.5,0);
	\draw (12.5,0)-- (12,0);
	\draw (13,0)-- (14,1);
	\draw (14,0.5)-- (13.5,0);
	\draw (13.5,0)-- (13,0);
	\draw [dotted] (0,1)-- (14.5,1);
	\draw [dotted] (0,2)-- (14.5,2);
	\draw [dotted] (1,3)-- (1,0);
	\draw [dotted] (2,3)-- (2,0);
	\draw [dotted] (3,3)-- (3,0);
	\draw [dotted] (4,3)-- (4,0);
	\draw [dotted] (5,3)-- (5,0);
	\draw [dotted] (6,3)-- (6,0);
	\draw [dotted] (7,3)-- (7,0);
	\draw [dotted] (8,3)-- (8,0);
	\draw [dotted] (9,3)-- (9,0);
	\draw [dotted] (10,3)-- (10,0);
	\draw [dotted] (11,3)-- (11,0);
	\draw [dotted] (12,3)-- (12,0);
	\draw [dotted] (13,3)-- (13,0);
	\draw [dotted] (14,3)-- (14,0);
	\draw [->] (-0.5,0) -- (14.5,0);
	\draw [->] (0,-0.5) -- (0,3);
	\draw (14.5,0.3) node {$x$};
	\draw (0,3.3) node {$t$};
	\end{tikzpicture}
	\caption{Evolution of $u$ considered in \autoref{prop first 1/2} and \autoref{prop second 1/2}
		built upon ABDF evolution of \autoref{fig:abdfdynamics}.
		Two shades of gray denote, as above, right and left quasi-particles.
		The dotted grid has side length $1$, so $u=\pm 2$ respectively on dark and light gray areas.}
	\label{fig:uevolution} 
\end{figure}

\subsection{Continuation Shortly Before Integer Times}
To continue the solution in time intervals $t\in\left[  t_0+\frac12,t_0+1\right]$ 
is somewhat trickier because of two phenomena:
coalescence of quasi-particles, and growth of new pairs. We are now at time $t_0+1/2$, namely
\begin{align*}
	u\left(  t_0+\frac12,x,\omega\right)  
	&=2\sum_{z\in MA^{+}\left(\theta,act,\omega,t_0\right)  }1_{\left\{  z+\frac12\leq x<z+1\right\}}\\
	&\quad -2\sum_{z'\in MA^{-}\left(  \theta,act,\omega,t_0\right)
	}1_{\left\{  z'-1\leq x<z'-\frac12\right\}  }.
\end{align*}
The continuation depends on this configuration and on the section of the noise%
\[
\left\{  \omega\left(  t_0+1,x\right)  ;x\in\Z\right\}  .
\]
Indeed, at time $t_0+1$ we could observe the result of arising pairs, as it
was above at time $t_0$. These pairs started existing at time $t_0%
+\frac12$.

Let us also write explicitly where we want to arrive at at time $t_0+1$: called%
\[
\left(  \theta',act'\right)  =\T_{\text{ABDF}}\left(
t_0,\omega,(\theta,act)  \right)
\]
we want to have%
\begin{align*}
u\left(  t_0+1,x,\omega\right)   & :=2\sum_{z\in M^{+}\left(  \theta
'\right)  }1_{\left\{  z\leq x<z+\frac12\right\}  }-2\sum_{z\in
M^{-}\left(  \theta'\right)  }1_{\left\{  z-\frac12\leq
x<z\right\}  }\\
& +2\sum_{z\in A\left(  \theta',act',\omega,t_0+1\right)
}\left(  1_{\left\{  z\leq x<z+\frac12\right\}  }-1_{\left\{  z-\frac
{1}{2}\leq x<z\right\}  }\right)  .
\end{align*}%

\begin{proposition}\label{prop second 1/2}
Given the sets defined above, let us introduce also
\begin{align*}
MA_{iso}^{+}\left(  \theta,act,\omega,t_0\right)   & =\left\{  z\in
MA^{+}\left(  \theta,act,\omega,t_0\right)  :z+1\notin C\left(
\theta,act,\omega,t_0\right)  \right\}, \\
MA_{iso}^{-}\left(  \theta,act,\omega,t_0\right)   & =\left\{  z\in
MA^{-}\left(  \theta,act,\omega,t_0\right)  :z-1\notin C\left(
\theta,act,\omega,t_0\right)  \right\}  .
\end{align*}

Consider the functions $u\left(  t_0+\frac{1}%
{2},x,\omega\right)  $ and $u\left(  t_0+1,x,\omega\right)  $ we just defined.
The following function $u\left(  t,x,\omega\right)  $, $t\in\left[
t_0+\frac12,t_0+1\right]  $, interpolates between them and is a weak
solution of Burgers' equation: we set, for $t\in\left[  t_0+\frac12,t_0+1\right] $,
\begin{equation*}
	u\left(  t,x,\omega\right)  =u_{iso}\left(  t,x,\omega\right)  +u_{coa}\left(
	t,x,\omega\right)  +u_{ari}\left(  t,x,\omega\right),
\end{equation*}
where $u_{iso}$ collects \emph{isolated} quasi-particles,
\begin{align*}
	u_{iso}\left(  t,x,\omega\right)   
	&=\sum_{z\in MA_{iso}^{+}\left(\theta,act,\omega,t_0\right)  }u_{iso}^{(z,+)}\left(t,x,\omega\right)\\
	&\quad +\sum_{z'\in MA_{iso}^{-}\left(  \theta,act,\omega,t_0\right)  }u_{iso}^{\left(  z',-\right)  }\left(t,x,\omega\right),\\
	u_{iso}^{(z,+)  }\left(  t,x,\omega\right)   
	&=2\cdot
	1_{\left\{  z+\left(  t-t_0\right)  \leq x<z+\left(  t-t_0\right)
		+\frac12\right\}  },\\
	u_{iso}^{\left(  z',-\right)  }\left(  t,x,\omega\right)   
	&=-2\cdot1_{\left\{  z'-\left(  t-t_0\right)  -\frac12\leq
		x<z'-\left(  t-t_0\right)  \right\}  },
\end{align*}
$u_{coa}$ the \emph{coalescing} quasi-particles,
\begin{align*}
	u_{coa}\left(  t,x,\omega\right)   
	& =\sum_{x_0\in C\left(  \theta,act,\omega,t_0\right)  }u_{coa}^{\left(  x_0\right)  }\left(t,x,\omega\right),\\
	u_{coa}^{\left(  x_0\right)  }\left(  t,x,\omega\right)   
	&=2\cdot 1_{\left\{  x_0-\left(  t_0+1-t\right)  \leq x<x_0\right\}  }
	-2\cdot1_{\left\{  x_0\leq x<x_0+\left(  t_0+1-t\right)  \right\}  },
\end{align*}
and finally, with $\left(  \theta',act'\right)  =\T_{\text{ABDF}}\left(  t_0,\omega,(\theta,act)  \right)  $,
$u_{ari}$ corresponds to \emph{arising} pairs,
\begin{align*}
u_{ari}\left(  t,x,\omega\right)   
& =\sum_{x_0\in A\left(  \theta^{\prime
},act',\omega,t_0+1\right)  }u_{ari}^{\left(  x_0\right)  }\left(
t,x,\omega\right),\\
u_{ari}^{\left(  x_0\right)  }\left(  t,x,\omega\right)   
& =2\cdot
1_{\left\{  x_0\leq x<x_0+\left(  t-t_0-\frac12\right)  \right\}
}-2\cdot1_{\left\{  x_0-\left(  t-t_0-\frac12\right)  \leq
x<x_0\right\}  .}
\end{align*}
\end{proposition}

\begin{proof}
The fact that, for almost all $x$, $\left(  u_{iso}+u_{coa}+u_{ari}\right)
\left(  t,x,\omega\right)  $ coincides with the functions $u\left(  t_0+\frac{1}%
{2},x,\omega\right)  $ and $u\left(  t_0+1,x,\omega\right)  $ defined above
can be easily checked -- notice that $u_{coa}\left(  t_0+1,x,\omega\right)
=0$ and $u_{ari}\left(  t_0+\frac12,x,\omega\right)  =0$. 
Considered by itself, each term of $u_{iso}$ is a weak solution on $\left[  t_0+\frac12%
,t_0+1\right]  $; similarly, each $u_{coa}^{\left(  x_0\right)  }\left(
t,x,\omega\right)  $ is a coalescing pair on $\left[  t_0+\frac12%
,t_0+1\right]  $ and each $u_{ari}^{\left(  x_0\right)  }\left(
t,x,\omega\right)  $ is an arising pair on $\left[  t_0+\frac12%
,t_0+1\right]  $. Thus the sum of all these functions is a weak solution if
they have disjoint supports.

Elements $u_{iso}^{(z,+)  }\left(  t,x,\omega\right)  $ have
supports of the form $[x\left(  t\right)  ,x\left(  t\right)  +\frac12)$
with $x\left(  t\right)  =z+\left(  t-t_0\right)  $; and $[x'\left(
t\right)  -\frac12,x'\left(  t\right)  )$ with $x'\left(
t\right)  =z'-\left(  t-t_0\right)  $ for $u_{iso}^{\left(
z',-\right)  }\left(  t,x,\omega\right)  $.
Supports of functions $u_{iso}^{(z,+)  }\left(  t,x,\omega
\right)  $ cannot intersect each other, since quasi-particles move in parallel;
the same holds for $u_{iso}^{\left(  z',-\right)  }\left(  t,x,\omega\right)$, among themselves.

The support of a function $u_{iso}^{(z,+)  }\left(
t,x,\omega\right)  $ cannot intersect the support of a function $u_{iso}%
^{\left(  z',-\right)  }\left(  t,x,\omega\right)  $ for the following
reason. It is not possible if $z\geq z'$, because they move in opposite directions. 
Keeping in mind that we consider $t\in\left[  t_0+\frac12,t_0+1\right]  $, 
the same argument applies when $z=z'-1$. 
If $z=z'-2$, then $x_0:=z+1$
is of class $C\left(  \theta,act,\omega,t_0\right)  $, hence $z $ and
$z'$ cannot be in $MA_{iso}^{+}\left(  \theta,act,\omega
,t_0\right)  $ and $MA_{iso}^{-}\left(  \theta,act,\omega,t_0\right)  $
respectively. It remains to discuss the case $z\leq z'-3$. But now the supports
$[x\left(  t\right)  ,x\left(  t\right)  +\frac12)$ and $[x^{\prime
}\left(  t\right)  -\frac12,x'\left(  t\right)  )$ do not have
sufficient time to meet, for $t\in\left[  t_0+\frac12,t_0+1\right]  $.
Summarizing, we have proved that all terms of $u_{iso}\left(  t,x,\omega
\right)  $ have disjoint supports.

Let $x_0\in C\left(  \theta,act,\omega,t_0\right)$. Its
corresponding coalescing pair $u_{coa}^{\left(  x_0\right)  }\left(
t,x,\omega\right)  $; its support has the form $[x_0-\left(  t_0%
+1-t\right)  ,x_0+\left(  t_0+1-t\right)  )$, contained in $[x_0%
-\frac12,x_0+\frac12)$ for $t\in\left[  t_0+\frac12%
,t_0+1\right]  $. 
These supports are clearly disjoint when $x_0$ varies in $C\left(\theta,act,\omega,t_0\right) $.
They are also disjoint from any element of $u_{iso}\left(  t,x,\omega\right) $:
let us see why, in
the case of a function $u_{iso}^{(z,+)  }\left(  t,x,\omega
\right)  $. Since $x_0\in C\left(  \theta,act,\omega,t_0\right)  $,
$x_0-1$ cannot be of class $MA_{iso}^{+}\left(  \theta,act,\omega
,t_0\right)  $. Thus we need to have $z<x_0-1$ and $u_{iso}^{\left(
z,+\right)  }\left(  t,x,\omega\right)  $ cannot reach the coalescing pair in
the time interval $\left[  t_0+\frac12,t_0+1\right]  $, for the same
reason why different points of $MA_{iso}^{+}\left(  \theta,act,\omega
,t_0\right)  $ cannot lead to intersections.

Finally, let us consider a point $x_0\in A\left(  \theta',act',\omega,t_0+1\right)$ 
where a new arising pair starts to exists at time $t_0+\frac12$,
and the associated function $u_{ari}^{\left(  x_0\right)  }\left(  t,x,\omega\right)  $. 
Let us first discuss the case of $z\in MA_{iso}^{+}\left(  \theta,act,\omega,t_0\right)$. 
If $z\geq x_0$ or $z\leq x_0-2$ there is no intersection: the
difficult case is $z=x_0-1$. But in such a case, having excluded by $MA_{iso}^{+}\left(  \theta,act,\omega,t_0\right)$ 
the possibility of coalescing points, we should have $\theta'(x_0)=1\neq 0$,
in contradiction with $x_0\in A(\theta',act',\omega,t_0+1)$.
Hence this case does not exist. Points $z\in MA_{iso}^{-}\left(
\theta,act,\omega,t_0\right)  $ are similar.

The hardest case is when $x_0\in A\left(  \theta',act',\omega,t_0+1\right)$ also belongs to $C\left(  \theta,act,\omega,t_0\right)  $. 
In plain words, the question is whether a pair may arise in a point of coalescence. 
This case is solved by \autoref{corollary coalescence arising}, implying that if 
$x_0\in C\left(  \theta,act,\omega,t_0\right)$, then
$x_0\notin A\left(  \theta',act',\omega,t_0+1\right)$. 
Indeed, assuming the former, by definition $x_0-1,x_0+1\in A(\theta,act,\omega,t_0)$,
and by \autoref{corollary coalescence arising} this implies $act'(x_0)=0$, 
so $x_0\notin A\left(  \theta',act',\omega,t_0+1\right)$.
This rules out the last possible intersection of supports, and the proof is complete.
\end{proof}

\subsection{Main Result}
Merging the statements of \autoref{prop first 1/2} and \autoref{prop second 1/2}, 
along with the simple claim of \autoref{prop theta from u}, we finally get the main result of this work:

\begin{theorem}\label{thm:mainresult}
Given, at time $t_0=0$, an element $(\theta,act)  \in\Lambda$
and the section $\left\{  \omega\left(  0,x\right)  ;x\in\Z\right\}
$, define $u_{0}\left(  x,\omega\right)  :=u\left(  0,x,\omega\right)  $,
following \autoref{Def u_0}.

Construct the stochastic process $\left(  \theta\left(  t_0,\omega\right)
,act\left(  t_0,\omega\right)  \right)  $, $t_0\in\N$, by setting%
\[
\left(  \theta\left(  t_0,\omega\right)  ,act\left(  t_0,\omega\right)
\right)  :=\phi_{\text{ABDF}}\left(  t_0,\omega\right)  \left(
\theta,act\right)
\]
namely by performing the ABDF random dynamics.

Define the stochastic process $u\left(  t,x,\omega\right)  $, $t\in
\lbrack0,\infty)$, $x\in\R$ as follows. For every $t_0\in\N%
$, define $u(t_0,x,\omega)  $ from \autoref{Def u_0}
with respect to $(\theta,act)  $ given by $\left(  \theta\left(
t_0,\omega\right)  ,act\left(  t_0,\omega\right)  \right)  $; define
$u\left(  t,x,\omega\right)  $ for $t\in\left[  t_0,t_0+\frac{1}%
{2}\right]  $ by \autoref{prop first 1/2}; finally define $u\left(
t,x,\omega\right)  $ for $t\in\left[  t_0+\frac12,t_0+1\right]  $ by
\autoref{prop second 1/2}.

Then $u\left(  t,x,\omega\right)  $ is a weak solution of Burgers' equation.
\end{theorem}

We have thus shown that, given a realization of the ABDF process, we construct
a weak solution of Burgers' equation; for almost every $\omega$, from this weak
solution it is possible to reconstruct the underlying ABDF realization, by
\autoref{prop theta from u}.

\begin{remark}
The stochastic process $u$ so defined is adapted to the noise filtration
shifted by $\frac12$. Namely, if $\F_{t}$ is the natural
filtration of the noise, $u\left(  t,\cdot\right)  $ is $\F_{t+\frac12}$-adapted, 
due to the creation mechanism that starts at
half-integer times. This anticipation is just instrumental, and not a deep phenomenon. 
One can develop an alternative construction such that $u\left(  t,\cdot\right)  $ is
$\F_{t}$-adapted, just shifting time by $\frac12$, or more precisely
starting to create new particles at integer times and completing annihilation
at half-integer times. However, we deem the construction just described more elegant.
\end{remark}

\section{Final Remarks on TASEP, Burger's Equation and KPZ Universality Class}

Among the most striking recent results on stochastic systems is the first quite complete understanding 
of the KPZ fixed point as a scaling limit of fluctuations of the height function associated to TASEP. 
Height functions of models in the KPZ universality class are conjectured to converge in the 1:2:3 scaling limit,
\begin{equation*}\label{123scaling}
h(t,x)\mapsto h^\epsilon(t,x)= \epsilon^{1/2} h (\epsilon^{-3/2}t,\epsilon^{-1}x)-C_\epsilon t, \quad \epsilon\downarrow 0,
C_\epsilon\uparrow\infty,
\end{equation*}
to a universal, scale invariant limit process, characterized as a Markov process 
by its transition probabilities in \cite{Matetski2016}.
Although we did not discuss scaling limits, it is essential to refer to
the recent works \cite{Nica2020,Quastel2019} on transition probabilities
of KPZ fixed point obtained as limits of the ones of TASEP, see also \cite{Arai2020} for the discrete-time setting.
The $1:2:3$ scaling we referred to above identifies fluctuations,
and it is worth recalling that macroscopic limits of models such as TASEP
are classically known to be solutions of nonlinear conservation laws, \cite{rost1981non,Rezakhanlou1991}.

The Kardar-Parisi-Zhang (KPZ) equation, introduced in \cite{kardar1986dynamic},
\begin{equation}\label{eq:kpzequation}
\partial_t h=\nu \partial_x^2 h+\lambda (\partial_x h)^2+\sigma \xi, \quad \nu,\lambda,\sigma>0,
\end{equation}
where $\xi$ denotes space-time white noise,
is not invariant under the 1:2:3 scaling. Its solution theory, initiated in \cite{Bertini1997},
has been the starting point of recent breakthrough developments in stochastic PDE theory, \cite{Hairer2013,Gubinelli2017}. Solutions to the KPZ equation are special models in the KPZ universality class,
as they are expected to describe the unique heteroclinic orbit between the KPZ fixed point and the Gaussian
Edwards-Wilkinson fixed point, \cite{Quastel2020}. 
Under the 1:2:3 scaling, the diffusion and noise terms of \eqref{eq:kpzequation}
vanish, formally leading to the Hamilton-Jacobi equation
\begin{equation}\label{eq:hamiltonjacobi}
\partial_t h=\lambda (\partial_x h)^2.
\end{equation}

This, informally, suggests that the KPZ fixed point can be understood as a (stochastic) solution to \eqref{eq:hamiltonjacobi},
corresponding to Burgers' equation \eqref{eq:burgers} with $\lambda=-1$ and $u=\partial_x h$.
However, as pointed out in \cite{Matetski2016}, entropy solutions to \eqref{eq:hamiltonjacobi} given by the Hopf-Lax formula,
\begin{equation*}
h(t,x)=\sup_y \pa{h(0,y)-\frac{(x-y)^2}{4\lambda t}},
\end{equation*}
are not suited to describe the KPZ fixed point. 
Indeed, entropy solutions would not preserve the regularity of Brownian motion, unlike the KPZ fixed point. 
In addition, since the KPZ fixed point has the space regularity of Brownian motion, the nonlinear term of \eqref{eq:hamiltonjacobi} is ill-posed. The possibility of a different kind of weak solutions to \eqref{eq:hamiltonjacobi} describing the KPZ fixed point was left open in \cite{Matetski2016}.

Our result might thus be regarded as hinting to a relation between the weak, non-entropic, intrinsically stochastic
solutions of Burgers' equation we built and linked with discrete-time TASEP, and the KPZ fixed point.
However, our arguments yield a bijection of models before any scaling limit is considered,
and even conjecturing how non-entropic, intrinsically stochastic Burgers' solution might describe --or simply relate to--
the KPZ fixed point, seems very difficult.
We do mention it since the problem of finding an equation satisfied by the KPZ fixed point remains completely open,
and non-entropic weak solutions might be the right objects to consider
(\emph{cf.} \cite{Bakhtin2018} and remarks in the introduction of \cite{Corwin2015}).

\section*{Acknowledgements} 
BG acknowledges support by the Max Planck Society through the Max Planck Research Group
\textit{Stochastic partial differential equations} and by the Deutsche Forschungsgemeinschaft (DFG,
German Research Foundation) through the CRC 1283 \textit{Taming uncertainty and profiting from
randomness and low regularity in analysis, stochastics and their applications}.


\end{document}